\documentclass[10.9pt,a4paper, reqno]{amsart}
\usepackage{a4wide}
\usepackage[utf8]{inputenc}
\usepackage[english]{babel}

\usepackage{amsfonts,amssymb,amsmath}

\usepackage{xcolor}
\usepackage[colorlinks,
    linkcolor={red!50!black},
    citecolor={blue!50!black},
    urlcolor={blue!80!black}]{hyperref}

\usepackage{tikz}
\usepackage[all]{xy}
\usepackage{enumitem}
\makeatletter
\newcommand{\mylabel}[2]{#2\def\@currentlabel{#2}\label{#1}}
\usetikzlibrary{calc, matrix, arrows, cd}

\newcommand{\bsm}{\left(\begin{smallmatrix}}
\newcommand{\esm}{\end{smallmatrix}\right)}
\newtheorem{innercustomthm}{Theorem}
\newenvironment{customthm}[1]
  {\renewcommand\theinnercustomthm{#1}\innercustomthm}
  {\endinnercustomthm}

\newtheorem{theorem}{Theorem}[section]

\newtheorem{corollary}[theorem]{Corollary}
\newtheorem{lemma}[theorem]{Lemma}
\newtheorem{proposition}[theorem]{Proposition}

\theoremstyle{definition}
\newtheorem{definition}[theorem]{Definition}

\newtheorem{example}[theorem]{Example}

\newtheorem{remark}[theorem]{Remark}
\newtheorem{notation}[theorem]{Notation}
\newtheorem{construction}[theorem]{Construction}
\newtheorem{claim}{Claim}
\newtheorem*{claim*}{Claim}

\newcommand{\sm}{\setminus}

\newcommand{\im}{\operatorname{im}}
\newcommand{\coker}{\operatorname{coker}}
\newcommand{\aug}{\operatorname{aug}}
\newcommand{\proj}{\operatorname{proj}}

\newcommand{\Tor}{\operatorname{Tor}}
\newcommand{\PD}{\operatorname{PD}}
\newcommand{\ev}{\operatorname{ev}}
\newcommand{\sn}{\operatorname{sn}}
\newcommand{\tr}{\operatorname{tr}}

 \definecolor{bettergreen}{rgb}{0.0, 0.5 0.0}

\newcommand{\Z}{\mathbb{Z}}
\newcommand{\Q}{\mathbb{Q}}

\newcommand{\C}{\mathbb{C}}

\newcommand{\Hom}{\operatorname{Hom}}
\newcommand{\Arf}{\operatorname{Arf}}
\newcommand{\ks}{\operatorname{ks}}
\newcommand{\id}{\operatorname{id}}

\newcommand{\fr}{\operatorname{fr}}

\title{Simply slicing knots}
\author[A.~Conway]{Anthony Conway}
\address{The University of Texas at Austin, Austin TX,  United States}
\email{anthony.conway@austin.utexas.edu}

\author[P.~Orson]{Patrick Orson}
\address{California Polytechnic State University, San Luis Obispo, CA, United States}
\email{patrickorson@gmail.com}

\author[M.~Pencovitch]{Mark Pencovitch}
\address{School of Mathematics and Statistics, University of Glasgow, United Kingdom}
\email{m.pencovitch.1@research.gla.ac.uk}

\begin{document}
\begin{abstract}
Given a simply-connected $4$-manifold with boundary the $3$-sphere, this paper establishes
sufficient conditions for a knot in the boundary to be sliced by a locally flat disc in the~$4$-manifold, whose complement has finite cyclic fundamental group.
In addition, necessary and sufficient conditions are described to ensure that such discs exist stably, that is after taking the connected sum of the~$4$-manifold with copies of~$S^2 \times S^2$.
\end{abstract}

\maketitle

\section{Introduction}

Given a~$4$-manifold~$N$ with boundary~$S^3$, one can ask which knots~$K\subset S^3$ bound locally flat slice discs~$D\subset N$ with a given disc group~$\pi_1(N\setminus D)$. 
When the disc group is abelian, we call~$D$ a \emph{simple} slice disc. 
Freedman proved that a knot~$K \subset S^3$ bounds a locally flat slice disc in~$D^4$ with disc group~$\Z$ if and only if it has Alexander polynomial one~\cite{Freedman,FreedmanQuinn}.
More generally, a knot~$K\subset  S^3$ admits such a slice disc in a simply-connected, compact~$4$-manifold~$N$ if and only if its Blanchfield form is presented by a size~$b_2(N)$ nondegenerate hermitian matrix~$A(t)$ such that~$A(1)$ represents the intersection form of~$N$~\cite{ConwayPiccirilloPowell}. 
We call a knot \emph{$G$-slice} in $N$ if it bounds a locally flat slice disc with disc group isomorphic to~$G$. In this article, we investigate which knots are~$\Z_d$-slice in simply-connected~$4$-manifolds, for~$d\neq 0$. 

The criteria we obtain to guarantee a simple slice disc exists will depend on the relative homology class of the hypothesised disc, so this article could also be thought as asking when a given nonzero class~$x \in H_2(N,\partial N)$ of divisibility~$d$ is represented by a simple slice disc (which will necessarily be a~$\Z_d$-disc) with boundary a given~$K$.
Our main results provide sufficient conditions for this to occur (Theorem~\ref{thm:Main}) as well as a necessary and sufficient conditions for this to occur stably, i.e.\ for there to exist an integer~$k\geq 0$ so that~$K$ bounds a disc~$D \subset N\#^k(S^2\times S^2)$ representing the class~$x \oplus 0 \in H_2(N\#^k(S^2\times S^2))$ (Theorem~\ref{thm:StableWithBoundaryIntro}).
In the case that $d$ is a prime power,  and provided a disc exists for some $k$, we give a formula computing the smallest such $k$ in terms of the algebraic topology of $N$ and the Levine-Tristram signature of $K$ (Corollary~\ref{cor:StabilisingNUmber}). 

Results of the type we obtain in this article, but realising given \emph{absolute} homology classes by \emph{closed} submanifolds, were proved in a series of papers by Lee and Wilczy\'nski~\cite{LWCommentarii,LWKtheory,LWGenus}. Our method of proof has a similar structural approach, and uses their algebraic splitting theorem~\cite[Theorem 6.1]{LWKtheory} as a crucial mechanism to guarantee the slice discs described above in the stable manifold~$N\#^k(S^2\times S^2)$ survive the process of destabilising back to~$N$.

\subsection{Simple slice discs}

Given a $4$-manifold $N$ with $\partial N\cong S^3$, and $x\in H_2(N,\partial N)$, we denote by $x \cdot x$ the algebraic self-intersection number of the unique class in~$H_2(N)$ that maps to $x$.

\medskip

Our first main result provides sufficient conditions for a knot to bound a simple disc representing a given homology class.

\begin{theorem}
\label{thm:Main}
Let \(N\) be a compact, oriented, simply-connected $4$-manifold with boundary \(S^3\). Let $x \in H_2(N, \partial N)$ be a nonzero class of divisibility $d$, and 
suppose \(K\subset S^3\) is a knot such that~\(H_1(\Sigma_d(K))=0\). 
The following are equivalent:
\begin{itemize}
\item The knot $K$ is sliced by a simple disc in $N$ representing $x$.
\item
\begin{enumerate}
\item \label{item:1Main} $\operatorname{Arf}(K)+\ks(N)+\frac{1}{8}(\sigma(N)-x \cdot x) \equiv 0 \ \text{ mod } 2$ \ if $x$ is characteristic, and
\item\label{item:2Main}  $b_2(N)\geq \max_{0\leq j<d}|\sigma (N)-\frac{2j(d-j)}{d^2}\, x\cdot x +\sigma_K (e^{\frac{2\pi i j}{d}})|$.
\end{enumerate}
\end{itemize}
\end{theorem}

Here, ~$\ks(N)$ and~$\sigma(N)$ respectively denote the Kirby-Siebenmann invariant and signature of~$N$,  whereas~$\sigma_K \colon S^1 \to \Z$ is the Levine-Tristram signature of~$K$.
For~$x$ characteristic,~$\sigma(N)-x \cdot x$ is divisible by~$8$ because the intersection form of $N$ is nonsingular; see e.g.~\cite[Proposition~1.2.20]{GompfStipsicz}.
A class $x \in H_2(N,\partial N)$ is \emph{characteristic} if $x \cdot a=a \cdot a$ mod $2$ for every $a \in H_2(N)$ and \emph{ordinary} otherwise.

\begin{remark}
The case $x=0$ is excluded from Theorem~\ref{thm:Main}, but we note that, since simple discs in $N$ have group $\Z$ if and only if they are nullhomologous (see e.g. \cite[Lemma~5.1]{ConwayPowell}),
the results of~\cite{ConwayPiccirilloPowell} already provide necessary and sufficient algebraic conditions for the existence of a simple slice disc.
\end{remark}

Over the past couple of years, a lot of attention has gone into investigating when a knot bounds a nullhomologous slice disc (i.e.\ when a knot is \emph{$H$-slice}) in a given $4$-manifold~\cite{ConwayNagel, ManolescuMarengonSarkarWillis,
ManolescuMarengonPiccirillo,
IidaMukherjeeTaniguchi,ManolescuPiccirillo,KjuchukovaMillerRaySakalli}. More recently, there has also been some interest in the case of nonzero classes~\cite{Ren,Qin}.
Most of these articles focus on smooth obstructions,  so Theorem~\ref{thm:Main} can be thought of as a topological counterpart to these recent results.

\medskip

\noindent We now highlight some examples and consequences of Theorem~\ref{thm:Main} for small values of $d$.

\subsubsection*{The case $d=1$}

In the case $d=1$ of Theorem~\ref{thm:Main}, or equivalently the case that $x$ is primitive, the hypotheses simplify considerably. The condition $H_1(\Sigma_d(K))=H_1(S^3)=0$ holds vacuously, as does the second condition, since it reduces to the inequality~$b_2(N) \geq |\sigma(N)|$.

\begin{corollary}
\label{cor:SimpleDisc}
Let \(N\) be a compact, oriented, simply-connected $4$-manifold with boundary~\(S^3\),  let~$x \in H_2(N, \partial N)$ be a nonzero primitive class, and let~\(K\subset S^3\) be a knot.
\begin{itemize}
\item For $x$ ordinary,  $K$ is sliced by a simple disc in $N$ representing $x$.
\item For $x$ characteristic,  $K$ is sliced by a simple disc in $N$ representing $x$ if and only if 
$$\operatorname{Arf}(K)+\ks(N)+\frac{1}{8}(\sigma(N)-x \cdot x) \equiv 0 \ \text{ mod } 2.$$
\end{itemize}
\end{corollary}

This corollary recovers~\cite[Corollary 1.15]{KasprowskiPowellRayTeichnerEmbedding}, which, given a closed simply-connected $4$-manifold $X$, concerns sliceness in $X^\circ:=X\setminus  \operatorname{Int}(B^4)$.
\begin{corollary}
\label{cor:KasprowskiPowellRayTeichnerEmbedding}
For~$X\neq S^4,\C P^2,*\C P^2$ simply-connected, every knot $K\subset \partial X^\circ$ is slice in~$X^\circ$.
\end{corollary}
\begin{proof}
We prove that~$H_2(X^\circ)$ contains a nonzero primitive ordinary class.
If $X$ is spin,  any nonzero primitive class is ordinary.
In the nonspin case,  since $X\neq \C P^2,* \C P^2$ by hypothesis, work of Freedman~\cite{Freedman} ensures that $b_2(X)>1$.
We therefore obtain a
 direct sum decomposition~$H_2(X^\circ)=\langle x,y\rangle \oplus H'$ 
for some abelian group~$H'$ and $x,y \in H_2(X^\circ)$,  and one verifies that among the primitive classes~$x,y$ and~$x+y$, at least one is ordinary.
The result now follows from Corollary~\ref{cor:SimpleDisc}, noting that $H_2(X^\circ)\cong H_2(X^\circ,\partial X^\circ)$.
\end{proof}
The proof of Corollary~\ref{cor:KasprowskiPowellRayTeichnerEmbedding} shows that every knot bounds a disc in~$X^\circ$ which represents a primitive homology class; the inclusion of the adjective ``primitive" can also be deduced from the proof of~\cite[Corollary 1.5]{KasprowskiPowellRayTeichnerEmbedding} .
As an example application, note that since~$1\in \Z \cong H_2(\C P^2)$ is characteristic, Corollary~\ref{cor:SimpleDisc} shows that a knot~$K$ bounds a primitive disc in~$(\C P^2)^\circ$ if and only if~$\Arf(K)=0$ mod $2$, recovering~\cite[Theorem 1.2]{MarengonMillerRayStipsicz}.


\begin{example}
Since $K3$ is spin, any primitive class $x \in H_2(K3^\circ,S^3)$ is ordinary, and therefore Corollary~\ref{cor:SimpleDisc} implies that any knot bounds a disc in $K3^\circ$ representing $x$.
This contrasts with the smooth category where the corresponding result is false; see e.g.~\cite[Section~\(4.4\)]{ManolescuMarengonPiccirillo}.
\end{example}

\subsubsection*{The case $d=2$}
Next we consider the~$d=2$ case of Theorem~\ref{thm:Main}.
In this case,  the condition~$H_1(\Sigma_2(K))=0$ is equivalent to $|\det(K)|=1$ (see e.g.~\cite[Corollary~9.2]{Lickorish}).
This implies that $\Delta_K(-1) \equiv \pm 1$ mod~$8$ and therefore $\Arf(K)=0$ (see e.g.~\cite[Theorem~10.7]{Lickorish}).

\begin{corollary}
\label{cor:d=2}
Let \(N\) be a compact, oriented, simply-connected $4$-manifold with boundary \(S^3\) and let $x \in H_2(N, \partial N)$ be a nonzero class of divisibility $2$. 
Suppose \(K\subset S^3\) is a knot such that~\(|\det(K)|=1\).
The following are equivalent:
\begin{itemize}
\item The knot $K$ is sliced by a simple disc in $N$ representing $x$.
\item 
\begin{enumerate}
\item $\ks(N)+\frac{1}{8}(\sigma(N)-x \cdot x) \equiv 0 \ \text{ mod } 2$ \ when $x$ is characteristic, and
\item $b_2(N)\geq \operatorname{max}(\,|\sigma (N)-\frac{1}{2}\,x\cdot x +\sigma(K)|\,,\,|\sigma(N)|\,)$. 
\end{enumerate}
\end{itemize}
\end{corollary}
Here,  $\sigma(K)=\sigma_K(-1)$ denotes the (Murasugi) signature of $K$.

\begin{example}
\label{ex:CP2d=2}
In~$(\C P^2)^\circ$, the second condition can be rewritten as~$1 \geq \operatorname{max}(|-1+\sigma(K)|,1)$.
Since~$2 \in \Z \cong H_2(\C P^2)$ is ordinary and the signature of a knot is even, it follows that~$K$ with~$|\det(K)|=1$ bounds a simple~$\Z_2$-disc in~$\C P^2$ if and only if~$\sigma(K) \in \{0,2\}$.
In~$(\overline{\C P^2})^\circ$, the second condition can be rewritten as~$1 \geq \operatorname{max}(|1+\sigma(K)|,1)$, so this time it follows that~$K$ with~$|\det(K)|=1$ bounds a simple~$\Z_2$-disc if and only if~$\sigma(K) \in \{-2,0\}$.
\end{example}

\begin{remark}
\label{rem:Trefoil}
The left handed trefoil knot $K$ has nontrivial determinant, and it is not too difficult to verify that it bounds a $\Z_2$-disc in $(\C P^2)^\circ$;  such a disc arises by unknotting $K$ via a single crossing change whose corresponding crossing change disc intersects $K$ twice algebraically, see e.g.~\cite[Figure 3.1]{MarengonMillerRayStipsicz}.
Thus, the assumption $|\det(K)|=1$ is not necessary for the conclusion of Corollary~\ref{cor:d=2} to hold,  and more generally, the assumption $H_1(\Sigma_d(K))=0$ is not necessary to the conclusion of Theorem~\ref{thm:Main}. However, as noted in Theorem~\ref{thm:StableWithBoundaryIntro} and Proposition~\ref{prop:SignatureCondition}, the conditions on the Arf invariant and the signature are necessary for the knot~$K$ to bound a disc $D$ representing $x$.
\end{remark}

\begin{remark}
Theorem~\ref{thm:Main}, its corollaries, and Theorem~\ref{thm:StableWithBoundaryIntro} also hold if $\partial N$ is an integral homology $3$-sphere provided the congruence in the characteristic case is replaced by
$$\operatorname{Arf}(K)+\ks(N)+\frac{1}{8}(\sigma(N)-x \cdot x) \equiv \mu(\partial N) \ \text{ mod } 2,$$
where~$\mu(\partial N)$ denotes the Rochlin invariant of~$\partial N$.
For example, when~$N=X_{\pm 1}(J)$ is the \(4\)-manifold obtained by attaching a \(\pm 1\)-framed \(2\)-handle to \(D^4\), we deduce that a knot~$K$ is sliced by a simple disc representing a primitive class if and only if \(\Arf(K) \equiv \Arf(J) \mod{2}\); indeed Gonz\'alez-Acu\~na proved that~$\mu(S^3_{\pm 1}(J))=\Arf(J)$~\cite{GonzalezAcuna}.
\end{remark}

\subsection{Simple slice discs stably}

We say a knot $K\subset N$ is \emph{stably slice} if it is the boundary of a locally flat disc in~\(N_k := N\#^k (S^2\times S^2)\) for some \(k \geq 0\).
The proof of Theorem~\ref{thm:Main} closely follows the approach of Lee and Wilczy\'nski~\cite{LWCommentarii,LWKtheory}:~the first condition of Theorem~\ref{thm:Main} ensures that $K$ is stably slice, whereas the second condition ensures that the \(S^2\times S^2\) summands can be removed without harming the disc. We record the outcome of this first step as it might be of independent interest.
We say that a homology class~$x \in H_2(N,\partial N)$ is \emph{stably representable} if there is a~$k \geq 0$ such that~$x \oplus 0 \in H_2(N_k,\partial N)$ can be represented by an embedded surface.

\begin{theorem}
\label{thm:StableWithBoundaryIntro}
Let~$N$ be a simply-connected~$4$-manifold with boundary~$S^3$, let~$K\subset S^3$ be a knot.
\begin{enumerate}
\item\label{item:intro1} The knot $K$ is stably sliced by a disc stably representing $x$ if and only if either~$x$ is  ordinary or~$x$ is characteristic and
\begin{equation}
\label{eq:ArfConditionmain}
\operatorname{Arf}(K)+\ks(N)+\frac{1}{8}(\sigma(N)-x \cdot x) \equiv 0 \mod2.
\end{equation}
\item\label{item:intro2} For any~$g > 0$, every class~$x\in H_2(N,\partial N)$ is stably representable by a genus $g$ surface with boundary $K$.
\color{black}
\end{enumerate}
In both \eqref{item:intro1} and \eqref{item:intro2}, the surface may be taken to be simple.
\end{theorem}

Given a class~$x \in H_2(N,\partial N)$, the combination of Theorem~\ref{thm:Main} and Theorem~\ref{thm:StableWithBoundaryIntro} suggests analysing the minimal number of stabilisations needed for a knot~$K$ to bound a disc representing~$x$.
When~$x=0$ and $N=D^4$, this quantity 
is referred to as the \emph{stabilising number} $\sn(K)$ and has been studied in~\cite{ConwayNagel,KonnoMiyazawaTaniguchiInvolutions,FellerLewarkBalanced,FukumotoTaniguchi}.

\begin{definition}
For $x \in H_2(N,\partial N)$, the \emph{$(x,N)$-stabilising number} of $K$ refers to
\[
 \sn_{x,N}(K)=\min \{\ k \ | \ K  \text{ bounds a disc \(D \subset  N\#^k (S^2\times S^2)\) stably representing } x \}
\]
whereas the \emph{simple $(x,N)$-stabilising number} is 
\[ \sn_{x,N}^{\operatorname{simple}}(K)=\min \{\ k \ | \ K \text{ bounds a simple disc \(D \subset  N\#^k (S^2\times S^2)\) stably representing } x \}.
\]
\end{definition}

\begin{remark}
Clearly, by definition, it is always the case that $\sn_{x,N}(K) \leq \sn_{x,N}^{\operatorname{simple}}(K)$.
\end{remark}

When~$N=D^4$ and~$x=0$, it is known that~$\sn(K)$ is finite if and only if~$\Arf(K)=0$~\cite{SchneidermanStable} (see also~\cite[Theorem 2.9]{ConwayNagel}) in which case~$\sigma_K(\omega) \leq \sn(K) \leq g_4(K)$, where the first inequality holds whenever~$\omega \in S^1$ is not the root of a~$p(t)\in \Z[t^{\pm 1}]$ with~$p(1)=1$~\cite[Theorems~3.10 and~5.15]{ConwayNagel}.
Here, $g_4(K)$ denotes the $4$-genus of $K$.
In this setting, requiring the disc be simple amounts to it having knot group~$\Z$ in which case,~$\sn_{0,D^4}^{\text{simple}}(K)$ is equal to the \emph{$\Z$-stabilising number}~$\sn_\Z(K)$ introduced in~\cite{FellerLewarkBalanced}.
Our results lead to analogous (but stronger) statements in the case where~${x \neq0}$,  as we are able to show that when $d$ is a prime power,  $\sn_{x,N}(K)$ only depends on the algebraic topology of $N$ and on the Levine-Tristram signature of $K$.

\begin{corollary}
\label{cor:StabilisingNUmber}
Fix a nonzero~$x \in H_2(N,\partial N)$ of divisibility $d$ 
and a knot~$K \subset \partial N$.
Then~$\sn_{x,N}(K)$ is finite if and only if~$\sn^{\operatorname{simple}}_{x,N}(K)$ is finite if and only if~$x$ is ordinary or~$x$ is characteristic and~\eqref{eq:ArfConditionmain} from Theorem~\ref{thm:StableWithBoundaryIntro} holds.

If the stabilising numbers are finite and also $H_1(\Sigma_d(K))=0$, then
\begin{equation}
\label{eq:max}
\sn_{x,N}^{\operatorname{simple}}(K)=\frac{1}{2}\left(\max_{0\leq j<d}\left |\sigma (N)-\frac{2j(d-j)}{d^2} \,x\cdot x+\sigma_K (e^{\frac{2\pi i j}{d}})\right |-b_2(N)\right).
\end{equation}
If, additionally, $d$ is a prime power, then
\[
\sn_{x,N}(K)=\sn_{x,N}^{\operatorname{simple}}(K).
\]
\end{corollary}
\begin{proof}
The characterisation of finite stabilisation numbers follows immediately from Theorem~\ref{thm:StableWithBoundaryIntro}. For the next statements, write~$m$ for the right-hand side of equation \eqref{eq:max}. Applying Theorem~\ref{thm:Main} to the~$4$-manifold~$N\#^m(S^2\times S^2)$ shows that if~$H_1(\Sigma_d(K))=0$, there exists a simple slice disc in~$N\#^m(S^2\times S^2)$ representing~$x\oplus 0$. Thus~$\sn_{x,N}^{\operatorname{simple}}(K)\leq m$ and hence~$\sn_{x,N}(K)\leq m$. The inequality~$\sn_{x,N}^{\operatorname{simple}}(K)\geq m$ follows from Proposition~\ref{prop:SignatureCondition}, which is essentially due to Gilmer~\cite{GilmerConfiguration}. As discussed in the proof of Proposition~\ref{prop:SignatureCondition}, Gilmer requires that~$d$ be a prime power dividing~$x$,
but when the surface is simple and $d$ is the divisibility of $x$,
one can argue that this stipulation can be dropped. 
Whence~$m$ is a lower bound for~$\sn_{x,N}^{\operatorname{simple}}(K)$, and also for~$\sn_{x,N}(K)$ if~$d$ is a prime power. This proves the two claimed equalities.
\end{proof}

\begin{remark}
We do not know whether there exist examples where $H_1(\Sigma_d(K))=0$ and $\sn_{x,N}(K)\neq\sn_{x,N}^{\operatorname{simple}}(K)$. 
It would be interesting to know whether such examples exist, or conversely whether prime power divisibility is a non-necessary condition, just an artefact of the proof.
\end{remark}

\subsection{Examples}

We give some sample applications for some of our results. 
First,  we give examples of knots for which Theorem~\ref{thm:Main} can be used to generate locally flat discs in a homology class, where a smooth disc can be proven to not exist.

\begin{example}
We show that if $K\in\{14n11942, 14n13659\}$ then $K$ bounds a locally flat~$\Z_3$-disc in~\((\overline{\C P}^2)^\circ\) but does not bound a smooth $\Z_3$-disc in~\((\overline{\C P}^2)^\circ\). To see $K$ does not bound a smooth disc, we used KnotJob~\cite{SchuetzKJ} to compute the $s$-invariant is $s(K)=8$, then applied~\cite[Theorem~1.3]{Qin} to the class $3\in H^2(\overline{\C P}^2,S^3)$, noting that $s(K)=8>3^2-3=6$. To see~Theorem~\ref{thm:Main} applies, we first computed that in each case $\left\vert\Delta_K(e^\frac{2\pi i}{3})\cdot\Delta_K(e^\frac{4\pi i}{3}) \right\vert =1$, so by a classical result of Goeritz and Fox~\cite{Goeritz34, Fox-56-iii} we have that $H_1(\Sigma_3(K))=0$. We then used Sage to compute~$\Arf(K)=1$ so that~Theorem~\ref{thm:Main}\eqref{item:1Main} holds. Finally, we used Sage to compute that in each case $\sigma_K (e^\frac{2\pi i}{3})=\sigma_K (e^\frac{4\pi i}{3})=-4$, from which it may be computed that~Theorem~\ref{thm:Main}\eqref{item:2Main} holds, confirming $K$ bounds the claimed locally flat disc.
\end{example}

\begin{remark}
Examples of the previous sort are difficult to find among low-crossing knots, as we now discuss. The knot invariants in Theorem~\ref{thm:Main} are the Arf invariant and Levine-Tristram signature, the latter of which achieves its maximum near the Murasugi signature $\sigma(K)=\sigma_K(-1)$, and computations with low-crossing knots show this is roughly where the quantity on the right of Theorem~\ref{thm:Main}\eqref{item:2Main} usually achieves its maximum. On the other hand, the current best techniques to obstruct smooth sliceness within a homology class use the $\tau$-invariant (see~\cite[Theorem 1.1]{MR2026543}) or~$s$-invariant (see~\cite{Qin}). Due to the well-known close connections between $\sigma(K)$, $\tau$, and~$s$, among low-crossing knots, there is a tension between trying to make $|s|$ or $|\tau|$ large relative to~$d$ to obstruct smooth sliceness, while desiring small $|\sigma(K)|$ relative to $d$ to apply~Theorem~\ref{thm:Main}.
\end{remark}

Next,  we consider when winding number 0 satellites of the negative torus knot~\(5_1=T(-2,5)\) bound a $\Z_d$-disc in~\(N=(\C P^2)^\circ\) and its stabilisations. For $d=3$, the knot $5_1$ is readily seen by \emph{ad hoc} methods to bound a \emph{smooth}~$\Z_d$-disc in~$(\C P^2)^\circ$; the purpose of taking satellites is that it is likely that a suitable choice of companion knot will preclude the satellite from being smoothly~$\Z_3$-slice in~$(\C P^2)^\circ$. As we shall see, the knot invariants used in~Theorem~\ref{thm:Main} are robust under a winding number $0$ satellite operation so these satellites will still bound locally flat $\Z_3$-discs.

\color{black}

\begin{example}
\label{ex:T25}
Use $P(K)$ to denote a satellite knot with pattern~$P=T(-2,5)$, winding number zero,  and companion $K$.
We claim~$P(K)$ bounds a~$\Z_d$-disc in~\(N=(\C P^2)^\circ\) if and only if~$d=3$.

We first show $P(K)$ indeed bounds a $\Z_3$-disc.
Performing winding number zero satellite operations does not affect the Levine-Tristram signature,  the Arf invariant and the homology of the~$d$-fold branched cover, so it suffices to verify the conditions for Theorem~\ref{thm:Main} are satisfied for~$P$.
For~$d$,~$p$ and~$q$ pairwise coprime, the~$d$-fold branched cover of the~$(p,q)$-torus knot is an integer Brieskorn homology sphere~\cite{MilnorBrieskorn}, so $H_1(\Sigma_3(P))=0$.
The remaining conditions for Theorem~\ref{thm:Main} follow because $P$ is \emph{smoothly} $\Z_3$-slice in~\(N=(\C P^2)^\circ\), but we check them manually as we will use the computation later anyway.

 The knot has~$\Arf(P)=1$ and~$\sigma(P)=4$, so the characteristic class~$3 \in H_2((\C P^2)^\circ,S^3)$ satisfies~\eqref{item:1Main} from Theorem~\ref{thm:Main}.
Using Sage, we compute~$\sigma_{P}(e^{\frac{2\pi i j}{3}})=0, 4, 4$ for~$j=0,1,2$, giving
\[\max_{0\leq j<3}\left|1-2j(3-j)+\sigma_{P} (e^{\frac{2\pi i j}{3}})\right|=1.\]
Thus~\eqref{item:2Main} from Theorem~\ref{thm:Main} is also satisfied and~$P(K)$ bounds a topological~$\Z_3$-disc, as claimed.

Now we  show that if $P(K)$ bounds a $\Z_d$-disc, then $d=3$.
Again, it suffices to verify this for $P$.
It follows from Theorem~\ref{thm:StableWithBoundaryIntro}~\eqref{eq:ArfConditionmain} that~$P$ does not bound any disc in~$N$ representing a primitive class, even stably.
Since~$\sigma(P)=~4$,  the knot~$P$ cannot be sliced by a disc representing a class of even divisibility $d$: if it did, Gilmer's inequality~$1 \geq |d^2/2-1-\sigma(P)|$ (see e.g.~\cite[Lemma~2.1 (2)]{MarengonMillerRayStipsicz}) would lead to a contradiction.
Similarly, Gilmer's work shows that if~$P$ were sliced by a disc representing a class of divisibility~$d$ divisible by an odd prime power~$m$, then~$1 \geq |(m^2-1)d^2/2m^2-1-\sigma_{P}(e^\frac{\pi i (m-1)}{m})|$ (see e.g.~\cite[Lemma~2.1~(3)]{MarengonMillerRayStipsicz} and~\cite[Theorem 3.6]{ManolescuMarengonPiccirillo}).
If~$d$ is a prime power, taking~$m=d$, this yields the inequality~$1 \geq |(d^2-1)/2-1-\sigma_{P}(e^\frac{\pi i (d-1)}{d})|$, which forces~$d=3$.
As noted above (and proved in Proposition~\ref{prop:SignatureCondition}), when obstructing~$P$ from bounding a~$\Z_d$-disc, the prime power hypothesis can be dropped, yielding the claim.

As a final, quick, sample application we consider for which $d>3$, coprime to $2$ and $5$, the stabilising number of $T(2,5)$ is finite.
 One can compute that for $d=7, 9$, item~\eqref{eq:ArfConditionmain} from Theorem~\ref{thm:StableWithBoundaryIntro} does not hold, and so \(T(2,5)\) does not bound a \(\Z_7\)- or \(\Z_9\)-disc stably. 
 When \(d=11\), item~\eqref{eq:ArfConditionmain} from Theorem~\ref{thm:StableWithBoundaryIntro} holds and so we may apply Corollary~\ref{cor:StabilisingNUmber} to compute
$$\sn_{11,(\C P^2)^\circ}(T(2,5))=\sn_{11,(\C P^2)^\circ}^{\operatorname{simple}}(T(2,5)) =27.$$
\end{example}


\subsection*{Organisation}

In Section~\ref{sec:Stable Embedding},  we show that the class $x$ can be stably embedded,  thus proving Theorem~\ref{thm:StableWithBoundaryIntro}.
In Section~\ref{sec:Sufficient}, we list three conditions guaranteeing a stable embedding can be promoted to an embedding.
In Section~\ref{sec:H2Free}-\ref{sec:EvenessCondition}, we verify the three conditions are satisfied under the hypotheses of Theorem~\ref{thm:Main}.
In Section~\ref{sec:ProofMain}, the strands of the argument are collected together to prove Theorem~\ref{thm:Main}.
The appendix proves two folklore results that are required during the proof.

\subsection*{Acknowledgements}
The authors are grateful to Ian Hambleton,  Daniel Kasprowski,  and Lisa Piccirillo for useful discussions as well as to Mark Powell and Aru Ray for several helpful suggestions and discussions concerning the material in Section~\ref{sec:Stable Embedding}.
AC was partially supported by the NSF grant DMS~2303674.

\subsection*{Conventions}
We work in the topological category. Embeddings are assumed to be locally flat and manifolds are assumed to be compact, connected and oriented.
Modules are assumed to be finitely generated.

\section{Stable embedding}\label{sec:Stable Embedding}

Throughout this section, let~$N$ be a simply-connected~$4$-manifold with boundary $S^3$ and set~$N_k:= N \#^k S^2 \times S^2$. We say that a  homology class~$x \in H_2(N,\partial N)$ is \emph{stably representable} if there is a~$k \geq 0$ such that~$x \oplus 0 \in H_2(N_k,\partial N)$ is represented by an embedded surface. The goal of this section is to prove Theorem~\ref{thm:StableWithBoundaryIntro}.

\medbreak

We briefly sketch the proof strategy. Without any stabilisation, one can always represent~$x$ by \emph{some} embedded surface. 
The objective is to perform ambient surgery on this surface to reduce the genus to zero, without leaving the homology class. Ambient surgery along curves in the surface needs to be guided  by ``surface-framed'', embedded discs in the $4$-manifold, whose interiors are disjoint from the surface. 
Without stabilising we will only have incorrectly framed, embedded discs whose interiors intersect the surface. Stabilising adds the freedom to change a disc framing by an even number (Lemma~\ref{lem:modifybyeven}), and to perform the stable Whitney move, removing algebraically cancelling double points between surface and disc (Lemma~\ref{lem:whitneytrick}).
 If $x$ is ordinary and $N$ is not spin, one also has the freedom to tube into a $w_2$-nontrivial sphere and change the framing on the disc. If $x$ is ordinary, the procedures just described mean stable ambient surgery can always be performed to reduce the surface to genus zero. 
 If~$x$ is characteristic, stable ambient surgery can be performed until the genus is reduced to one. There is then a genuine obstruction to performing the final surgery, measured by the condition~\eqref{eq:ArfConditionmain}. 
 Finally, in either case and at the expense of further stabilisations, surgeries may be performed on the ambient manifold, in the complement of the surface, to kill commutators in the  fundamental group and produce a \emph{simple} surface.

\begin{remark}
This strategy is inspired by the work of Lee and Wilczy\'nski in the closed case~\cite[proof of Theorem 2.1]{LWCommentarii} which itself draws on ideas from of Freedman and Kirby~\cite{FreedmanKirby}.
We adapt these arguments to the case of discs, referring to~\cite{LWCommentarii} when the arguments are unchanged, while also supplementing some details related to the verification that the disc lies in~$N \#^k S^2 \times S^2$ (i.e. that no $\# S^2 \widetilde{\times} S^2$ is involved) and that it indeed represents~$x \oplus 0$; see Proposition~\ref{prop:NullhomologousStablyRepresented}.
\end{remark}
\color{black}

\subsection{Disc framing obstruction and the Matsumoto-Freedman-Kirby quadratic form}

The first half of the proof of  Theorem~\ref{thm:StableWithBoundaryIntro} consists of arguing that it can be reduced to the following statement, which we call the \emph{ambient surgery criteria}. The purpose of the current section is to describe and motivate the quadratic form $q_F$ mentioned in the ambient surgery criteria.

\begin{theorem}[The ambient surgery criteria]
\label{thm:AmbientSurgeryCriterion}
Let~$N$ be a simply-connected~$4$-manifold with boundary~$S^3$, let~$K  \subset \partial N$ be a knot,  let~$x \in H_2(N,\partial N)$,  let $F$ be a surface representing $x$,  with boundary~$K$, and let $\gamma \subset F$ be a simple closed curve.
If
\begin{enumerate}
\item either $x$ is ordinary, or
\item $x$ is characteristic and $q_F([\gamma])=0$,
\end{enumerate}
then there is an $r \ge 0$ such that ambient surgery on $F \subset N_r$ along $\gamma$ is possible and such that the effect $F'$ of surgery represents $x \oplus 0$ and has $\partial F'=K$.
\end{theorem}

We emphasise that this theorem and its proof rely heavily on work of Freedman-Kirby~\cite{Freedman} and Lee-Wilczy\'nski~\cite{LWCommentarii} in the closed case.

Let $F \subset N$ be an embedded surface in a simply-connected $4$-manifold, and let $\gamma \subset F$ be a simple closed curve.
In order to perform surgery on $F$ along $\gamma$ \emph{ambiently}, we need to find an embedded disc $D \subset N$ that satisfies~$\partial D= F \cap D=\gamma$ and such that a certain framing condition, that we now describe, is~satisfied.

\begin{construction}[The framing obstruction $W(D) \in \Z$]
\label{cons:qF}
Let $\gamma \subset F$ be an embedded circle.
As~$H_1(N)=0$, this circle bounds an immersed disc~$D \looparrowright N.$
It can be assumed that the interior of $D$ meets $F$ transversally.
Note that, as $F$ is orientable, the line bundle $\nu(\gamma \subset F)$ is trivial; 
it thus admits a unique trivialisation up to homotopy (because trivialisations of a line bundle over a circle are classified by~$[S^1,O(1)]=\pi_1(O(1))=\pi_1(\{ \pm 1\rbrace)=1$).
Decompose the restriction of the normal $2$-plane bundle~$\nu(D \subset N)$ to $\gamma=\partial D$ as
$$ \nu(D \subset N)|_{\gamma} \cong \nu(\gamma \subset F) \oplus \xi$$
for some line bundle $\xi$ over $\gamma$.
As $\nu (\gamma\subset F)$ is trivial and $\nu(D\subset N)|_\gamma$ is trivial, so $\xi$ must also be trivial.
Thus $\xi$ admits a unique trivialisation up to homotopy.
Combining the framings of $\nu(\gamma \subset F)$ and~$\xi$ we therefore obtain a framing of $\nu(D \subset N)|_{\gamma}$.
Consider the Euler number obstruction to extending this \emph{surface framing} of $\nu(D \subset N)|_{\gamma}$ to a framing of~$\nu(D \subset N)$:
$$W(D) \in H^2(D,\partial D;\Z) \cong \Z.$$
The integer $W(D)$ is measuring the difference between the surface framing of~$\nu(D \subset N)|_{\gamma}$ and the framing of $\nu(D \subset N)|_{\gamma}$ induced by the unique trivialisation of the $2$-plane bundle $ \nu(D \subset N)$.
Observe that if the interior of $D$ is disjoint from $F$, i.e.~$\gamma=\partial D=D \cap F$, then $W(D)$ is precisely the obstruction to performing ambient surgery on~$F$ along~$\gamma$ using~$D$, with the result still an embedded surface.
\end{construction}

%

We now discuss the possibility of eliminating the framing obstruction of a disc $D$, and making its interior disjoint from $F$, at the expense of stabilising the ambient manifold.

\begin{remark}
\label{rem:BoundaryTwist}
A boundary twist of $D$ (see~\cite[p.~15]{FreedmanQuinn}) simultaneously changes both the algebraic intersection number~$D \cdot F$ and $W(D)$ by $\pm 1$; see e.g.~\cite[page 87]{FreedmanKirby}.
Thus we can always perform boundary twists so that one of the two quantities vanishes.
We will see below that, at the expense of stabilising $N$, the Whitney move can be performed (see Lemma~\ref{lem:whitneytrick}), so that if the disc and the surface meet algebraically 0 times, these intersections can be removed. We will also see, again at the expense of stabilising $N$, that the framing obstruction $W(D)$ can always be changed by $\pm 2$ (see Lemma~\ref{lem:modifybyeven}). Thus, if we allow stabilising, the obstruction to doing ambient surgery on $F$ along $\gamma$ to obtain an embedded result is whether or not $W(D)$ and $D \cdot F$ have the same parity. In other words, we need to consider the parity of $W(D)+D \cdot F.$
Given an embedded sphere $S\subset N$, defining $D':=D\# S$, we have 
$$W(D')=W(D)+Q_N([S],[S])\qquad \text{ and } \qquad D' \cdot F=D \cdot F+Q_N^{\partial}([F],[S]),$$
where $Q_N^\partial$ denotes algebraic intersections of relative and absolute classes,
and therefore
$$ W(D')+D' \cdot F=\left( W(D)+D \cdot F \right)+Q_N([S],[S])+Q_N^{\partial}([F],[S]).$$
When $F$ is characteristic $Q_N([S],[S])=Q_N^{\partial}([F],[S])$ mod $2$, and so replacing $D$ by $D'$ cannot change the overall parity. 
Thus in the characteristic case this will lead to a genuine obstruction.
However, when $[F]$ is ordinary, one can hope that the parity of $W(D)+D \cdot F$ can be changed by tubing into an appropriate sphere in $S\subset N$.
 \end{remark}

This obstruction in the characteristic case is captured by the Arf invariant of a certain quadratic form over $\Z_2$, whose definition we now recall.

\begin{proposition}[Matsumoto, Freedman-Kirby]
Let $N$ be a simply-connected $4$-manifold, let~$F \subset N$ be a properly embedded characteristic surface and let $\gamma\subset F$ be an embedded circle.
The assignment
$$q_F(\gamma)=:W(D)+D \cdot F \quad \text{mod } 2$$
gives rise to a well defined quadratic refinement $q_F\colon H_1(F;\Z_2)\to\Z_2$ of the intersection form on~$H_1(F;\Z_2)$.
\end{proposition}

\subsection{Reduction of Theorem~\ref{thm:StableWithBoundaryIntro} to the ambient surgery criteria}

The goal of this section is to prove Theorem~\ref{thm:StableWithBoundaryIntro}, assuming Theorem~\ref{thm:AmbientSurgeryCriterion} is true.

\begin{construction}[The framing $\fr_D$ of the boundary of an immersed disc $D \looparrowright N$]\label{constr:frD}
Given an immersed disc $D \looparrowright N$, the normal bundle $\nu(D)$ is trivial and,  up to homotopy, admits a unique framing compatible with its orientation.
This framing induces a framing of the rank~$2$ subbundle~$\nu(D)|_{\partial D}\subset \nu(\partial D \subset N)$.
This latter framing can be completed to a framing $\fr_D $ of the rank $3$ bundle~$\nu(\partial D \subset N)$ using the outwards normal in first position.
\end{construction}

The following result describes how a loop in $N$ can be framed so that the effect of surgery on it is a stabilisation of $N$ by $S^2 \times S^2$. 
In the presence of a surface $F \subset N$ representing a homology class $x$,  this result also provides a condition ensuring that $F$ represents $x \oplus 0$ in the stabilisation.

\begin{proposition}
\label{prop:NullhomologousStablyRepresented}
Let~$N$ be a simply-connected
$4$-manifold.
\begin{itemize}
\item Suppose $\alpha\subset N$ is an embedded loop that bounds an immersed disc $D$. Then surgery on~$N$ along $\alpha$ using the framing $\fr_{D}$ is homeomorphic to $N\#(S^2\times S^2)$.
\item Let~$F \subset N$ be a properly embedded surface representing a homology class $x \in H_2(N,\partial N)$, suppose that~$\alpha\cap F=\emptyset$ and that $\alpha \subset N \setminus F$ is nullhomologous.
Then a disc $D\subset N$ (and thus the framing $\fr_{D}$) may be chosen so that~$F$ represents~$x \oplus 0 \in H_2(N \# S^2 \times S^2, \partial N)$ in the effect of surgery.
\end{itemize}
\end{proposition}

\begin{proof}
Make the disc~$D$ embedded by using finger moves on~$D$ to push self-intersections of~$D$ past the boundary~$\partial D$; see e.g.~\cite[Figure 11.4]{DETBook}.
Finger moves on a disc do not change the framing, relative to a given boundary framing, so the framing~$\fr_{D}$ is unchanged. 
We can assume that a tubular neighbourhood~$\overline{\nu}(\alpha)$ has been chosen small enough so that~$\overline{\nu}(\alpha)\cap D$ is a closed collar of $\partial D$.
As~$D$ is embedded, so is the disc~$D':=D\setminus\nu(\alpha)$. 
Note that $\partial D'=D' \cap \overline{\nu}(\alpha)$. We set $\alpha':=\partial D'$.

The idea for showing the effect of surgery along $\alpha$ is $N\#S^2\times S^2$ is to use the framing $\fr_{D}$ to work in a preferred model of~$N \cong N \# S^4$ with~$\alpha \subset S^4$ in which it will be clear that surgery along~$\alpha$ results in a~$S^2 \times S^2$ summand instead of an~$S^2 \widetilde{\times} S^2$ summand. 
Use~$\fr_{D}$ to frame~$\alpha$ and thus identify~$\overline{\nu}(\alpha)=S^1\times D^3$.
Since~$\alpha'=\partial\overline{\nu}(\alpha) \cap  D'$,
by definition of~$\fr_{D}$, we have that~$\alpha'$ is a pushoff using the framing~$\fr_{D}$. We have hence identified~$\alpha'=S^1\times\{p\}$ for some~$p\in S^2$.
Consequently, we obtain an identification~$\overline{\nu}(\alpha)\cup \overline{\nu}(D') =S^1\times D^3\cup D^2\times D^2$,
the effect of attaching a~$4$-dimensional~$2$-handle to~$\partial(S^1\times D^3)=S^1\times S^2$ along the curve~$\alpha'=S^1\times\{p\}$.
Note this effect is homeomorphic to~$D^4$ and so we obtain a particular choice of homeomorphism~$\phi'_{D'}\colon N\# S^4\cong N$.
Excising~$S^1\times \mathring{D^3}$ and gluing back~$D^2\times S^2$ via the identity map is precisely doing surgery on~$N$ along~$\alpha$ with the given framing~$\fr_{D}$. In the punctured~$S^4$ summand, we obtain~$D^2\times S^2\cup D^2\times D^2=S^2\times S^2\sm (\mathring{D}^2\times \mathring{D}^2)$,  which is a punctured~$S^2\times S^2$.
Hence we obtain a homeomorphism~$\phi_{D'}\colon N\#(S^2\times S^2)\cong N'$. This completes the proof of the first item. For later on, we observe that, by construction, the sphere~$S^2\times\{p\}$ is~$D'\cup_{\alpha'} \Delta$, where~$\Delta:=D^2\times\{p\}\subset D^2\times S^2$ is a core of surgery.

We discuss the strategy for the second item before going into the details. If $\alpha$ is endowed with any framing $\fr_{D}$, as in the first item, then in the effect of surgery,~$F$ intersects~$\lbrace \operatorname{pt} \rbrace \times S^2$ zero times geometrically.
Thus the proof would be complete if the framing $\fr_{D}$ was such as to guarantee that a homology representative of~$[S^2 \times \lbrace \operatorname{pt} \rbrace]$ intersects~$F$ algebraically zero times.
For this,  we will construct $D$ with the additional property that the resulting~$D'$ intersects~$F$ algebraically $0$ times. 
The embedded sphere~$S=D'\cup\Delta$, where~$\Delta$ is a core of surgery, represents~$S^2 \times \lbrace \operatorname{pt} \rbrace$. 
As~$\Delta$ is a core of surgery, it does not intersect~$F$ (geometrically). 
Thus we will be able to conclude that~$S\cdot F=0$.

We construct the disc~$D \subset  N$ by first performing ambient surgery on a surface with boundary~$\alpha$; here are the details.
As~$\alpha$ is nullhomologous in~$N\sm F$,  choose a surface $\Sigma\subset N\sm F$ with~$\partial \Sigma=~\alpha$.
Since~$N$ is simply-connected, loops representing generators of~$H_1(\Sigma)$ bound immersed discs in~$N$; note that these discs might intersect~$F$, but by general position, and taking the tubular neighbourhood small, we may assume they do not intersect~$\overline{\nu}(\alpha)$.
By performing boundary twists on the immersed discs, we arrange that their framing obstructions vanish; recall Construction~\ref{cons:qF} and Remark~\ref{rem:BoundaryTwist}.
Performing ambient surgery on~$\Sigma$ using these discs leads to an immersed disc~$D\subset N$ with~$\partial D=\alpha$. 
The immersed disc~$D$ intersects~$F$ zero times algebraically:
upon performing surgery on~$\Sigma$, the discs that intersected~$F$ are replaced by pairs of discs that
intersect~$F$ in cancelling signs. Use finger moves on~$D$ to push self-intersections of the disc across the boundary~$\alpha$, thus making an embedded disc, still called~$D$.

We now have an embedded disc~$D$ with boundary~$\alpha$, intersecting~$F$ algebraically $0$ times. 
As before, by taking~$\overline{\nu}(\alpha)$ small enough, we may assume~$\overline{\nu}(\alpha)\cap D$ is a small closed collar of~$\partial D$. 
As before, remove the (open) collar, to obtain~$D':=D\setminus \nu(\alpha)$. As desired,~$D'$ intersects~$F$ algebraically~$0$ times. 
This completes the proof.
\end{proof}

We conclude this section by explaining how Theorem~\ref{thm:StableWithBoundaryIntro} follows from the ambient surgery criteria.

\begin{proof}[Proof of  Theorem~\ref{thm:StableWithBoundaryIntro} assuming Theorem~\ref{thm:AmbientSurgeryCriterion}.]
We claim that it suffices to prove that~$x$ is representable by a disc with boundary~$K$ if~$x$ is ordinary or if~$x$ is characteristic and~\eqref{eq:ArfConditionmain} holds, and by a punctured torus with boundary~$K$ otherwise.
First,  observe that trivial stabilisations of a surface allow us to raise the genus without affecting the homology class.
The claim will then follow once we prove that if a homology class is represented by a surface~$F \subset N$,  then it is stably represented by a simple surface of the same genus.
Represent normal generators of the commutator subgroup of~$\pi_1(N\setminus F)$ by simple closed curves.
Endow these curves with the framing specified by the second item of 
Proposition~\ref{prop:NullhomologousStablyRepresented}. This proposition then ensures that after surgery along these curves,~$x$ is now stably represented by a simple surface of the same genus as~$F$
(noting, in order to use~Proposition~\ref{prop:NullhomologousStablyRepresented}, that curves in the commutator subgroup are nullhomologous).
This concludes the proof of the claim.

We now use Theorem~\ref{thm:AmbientSurgeryCriterion} to prove the assertion that~$x$ is representable by a disc with boundary~$K$ if~$x$ is ordinary or if~$x$ is characteristic and~\eqref{eq:ArfConditionmain} holds, and by a punctured torus with boundary~$K$ otherwise.
Represent~$x \in H_2(N,\partial N)$ by a surface~$F$ with boundary~$K$.
Note that~\cite[Theorem 2]{Klug} implies that if~$x$ is characteristic, then
$$\operatorname{Arf}(K)+\ks(N)+\frac{1}{8}(\sigma(N)-x \cdot x)=\operatorname{Arf}(q_F) \in\Z/2.$$
Represent half of a symplectic basis of generators of~$H_1(F)$ by loops~$\gamma_1,\ldots,\gamma_r$ such that the~$\gamma_i$ are pairwise disjoint (this is possible for any algebraic symplectic basis; see e.g.~\cite[Third proof of Theorem~6.4]{FarbMargalit}).
\color{black}
If~$x$ is characteristic,  one can additionally assume that~$q_F(\gamma_i)=0$ for~$i=1,\ldots,r$ when~$\operatorname{Arf}(q_F)=0$ and~$q_F(\gamma_i)=0$ for~$i=1,\ldots,r-1$ when~$\operatorname{Arf}(q_F)=1$; see~\cite[Prop III.1.7, p.~55]{BrowderSurgery}.
Theorem~\ref{thm:AmbientSurgeryCriterion} ensures one can stably perform ambient surgery on~$\gamma_1$ in order to obtain a surface~$F_1$ stably representing~$x$ with~$g(F_1)=g(F)-1$.
Note that since the~$\gamma_i$ are pairwise disjoint,~$\gamma_2,\ldots,\gamma_r$ persist in~$F_1$ and satisfy~$q_{F_1}([\gamma_i])=q_{F}([\gamma_i])$.
The assertion now follows by induction and thus so does the theorem.
\end{proof}

\subsection{Preliminary lemmas before proof of the ambient surgery criteria}

Having reduced Theorem~\ref{thm:StableWithBoundaryIntro} to Theorem~\ref{thm:AmbientSurgeryCriterion}, we prove some preliminary lemmas before completing the proof of the latter.
The overall strategy is inspired by~\cite[Proof of Theorem 2.1]{LWCommentarii}.

\medbreak

The next lemma would also follow directly from~\cite[Theorem 1.1]{LWCommentarii}, but we include an independent proof, as our argument (inspired by~\cite[Proof of Corollary 1.15]{KasprowskiPowellRayTeichnerEmbedding}) is short.

\begin{lemma}
\label{lem:TopologicalWall}
Let~$N$ be a simply-connected~$4$-manifold whose boundary is either empty or a homology sphere.
Any ordinary primitive class~$x \in H_2(N)$ is
 representable by an embedded sphere.
\end{lemma}
\begin{proof}
Represent~$x \in H_2(N) \cong \pi_2(N)$ by an immersed sphere.
Add local cusps to this sphere so that the outcome, denoted~$S$, satisfies~$\mu(S)=0$; cusp moves are homotopies so~$S$ still represents~$x$.
Here,~$\mu(S) \in \Z$ denotes the signed counted of self-intersections of~$S$.
Since~$x \in H_2(N)$ is ordinary,  there exists an~$a \in H_2(N)$ with~$Q_N(x,a) \neq Q_N(a,a)$ mod~$2$,  from which it follows that~$S$ is not~$s$-characteristic (here, a properly immersed surface~$F$ in a~$4$-manifold~$M$ is \emph{$s$-characteristic} if~$Q_M^{\partial}([F],a) = Q_M(a,a)$ for every~$a \in \im(\pi_2(M) \to H_2(M) \to H_2(M;\Z_2))$; see e.g.~\cite[Definition 5.2]{KasprowskiPowellRayTeichnerEmbedding}).
Since~$x$ primitive and~$N$ is simply-connected,~$S$ is~$\pi_1$-negligible (meaning that~$\pi_1(N\setminus S) \to \pi_1(N)$ is an isomorphism) and we can therefore apply Stong's correction~\cite[Correction]{StongCorrection105} to~\cite[Theorem~10.5]{FreedmanQuinn} to deduce that~$S$ is homotopic to an embedding~$S'$ with~$x=[S]=[S'].$
\end{proof}

Next,  we recall the results mentioned in Remark~\ref{rem:BoundaryTwist}, which allow us to modify the framing obstruction and perform the Whitney trick, but in both cases at the expense of stabilising the ambient manifold.

\begin{lemma}[Stably modifying the framing obstruction by an even number]
\label{lem:modifybyeven}
Let~$N$ be a simply-connected $4$-manifold.
Suppose~$D\subset N$ is a generically immersed disc with embedded boundary~$\partial D\subset F$ and~$W(D)=e\in\Z$. 
Let~$f\in 2\Z$ be any even integer. Then there exists a generically immersed disc~$D'\subset N\#S^2\times S^2$, agreeing with~$D$ outside of~$S^2\times S^2$, and with~$W(D')=e+f$.
\end{lemma}
\begin{proof}
The proof can be found on~\cite[bottom of page 292]{LWCommentarii}, but goes back to~\cite[page 93]{FreedmanKirby}; see also~\cite[page 150]{Scorpan}.
\end{proof}

\begin{lemma}[Stable Whitney Trick]\label{lem:whitneytrick}
Let $N$ be a simply-connected $4$-manifold.
 Suppose~$S\subset N$ is an immersed surface and that there exist oppositely signed self-intersection points~$p,q\in S$ paired by a Whitney circle.
Then, in some~$N\#^k(S^2\times S^2)$, there is a framed embedded Whitney disc pairing~$p$ and~$q$.
\end{lemma}
\begin{proof}
The argument is well known; see e.g.~\cite[page 150]{Scorpan}:
eliminate self-intersections of an immersed Whitney disc by tubing into $S^2 \times S^2$ and then apply Lemma~\ref{lem:modifybyeven} to adjust the framing of the resulting embedded Whitney disc.
\end{proof}

\begin{corollary}
\label{cor:StableWhitney}
Let $N$ be a simply-connected $4$-manifold with boundary $S^3$.
Let~$F \subset N$ be a properly embedded surface representing a class $x \in H_2(N,\partial N)$. 
Let~$D \subset N$ be a disc that intersects~$F$ algebraically zero times.
Then there is some~$N \#^{k}(S^2 \times S^2)$ in which $F$ represents the class~$x\oplus 0$, and in which the disc~$D$ is regularly homotopic rel.~boundary to a disc~$D'$ that is disjoint from~$F$.
\end{corollary}
\begin{proof}
Apply Lemma~\ref{lem:whitneytrick} and the Whitney trick iteratively 
to the surface~$S=F\cup D$. 
Before performing the Whitney trick,  $F$ represents $x \oplus 0$ in $N \# ^k(S^2\times S^2)$ and the Whitney trick does not affect this.
\end{proof}

\color{black}

\subsection{Proof of the ambient surgery criteria}

We now prove the ambient surgery criteria whose statement we recall for the convenience of the reader.

\begin{customthm}{\ref{thm:AmbientSurgeryCriterion}}
\label{thm:AmbientSurgeryCriterionMain}
Let~$N$ be a simply-connected~$4$-manifold with boundary~$S^3$, let~$K \subset \partial N$ be a knot,  let~$x \in H_2(N,\partial N)$,  let $F$ be a properly embedded surface representing $x$ with boundary~$K$, and let~$\gamma \subset F$ be a simple closed curve.
If
\begin{enumerate}
\item either $x$ is ordinary, or
\item $x$ is characteristic and $q_F([\gamma])=0$,
\end{enumerate}
then there is an $r \ge 0$ such that ambient surgery on $F \subset N_r$ along $\gamma$ is possible and such that the effect $F'$ of surgery represents $x \oplus 0$ and has $\partial F'=K$.
\end{customthm}
\begin{proof}
The argument is inspired by~\cite[Proof of Theorem 2.1]{LWCommentarii}:
in both cases one wants to find an embedded disc~$D$ in a stabilisation of~$N$ in which~$D \cap F=\partial D=\gamma$ and~$W(D)=0$.
Once these two conditions are satisfied, one can perform ambient surgery.

We assert that there is an embedded disc $D$ in a stabilisation of~$N$ with~$\partial D=D \cap F=\gamma$; in this stabilisation $F$ represents $x \oplus 0$.
    We may push~$\gamma$ off~$F$ so that the pushoff~$\gamma'$ is nullhomologous in~$N\sm F$. 
    We briefly justify this claim. Choose an arbitrary pushoff of~$\gamma$; that is, pick an arbitrary section~$s\colon \gamma\to S(\nu(F\subset  N))\cong \gamma\times S^1$ of the sphere bundle of the (trivial) normal bundle of~$F\subset N$. As~$N$ is simply-connected,~$H_1(N\setminus F)$ is cyclic, generated by the class~$\mu$ of a meridian to~$F$. 
Thus~$[s(\gamma)]=k\mu$ for some~$k\in\Z$. Adding~$-k$ meridional turns to~$s$ produces a new section~$s'$, and~$s'(\gamma)=:\gamma'$ is then nullhomologous.

Endow the push off~$\gamma'$ with the framing specified by the second item of 
Proposition~\ref{prop:NullhomologousStablyRepresented}. By that proposition, after performing surgery on~$N$ along~$\gamma'$, we obtain an embedded disc in a stabilisation of~$N$ with~$\partial D=D \cap F=\gamma$.
Since we used a nullhomologous pushoff,  the surface represents~$x \oplus 0$; recall Proposition~\ref{prop:NullhomologousStablyRepresented}.
This establishes the assertion.

We begin with an outline of the argument to come.
In the characteristic case,  the condition~$q_F([\gamma])=0$ ensures that~$W(D)$ is even; after further stabilisations Lemma~\ref{lem:modifybyeven} ensures we can assume~$W(D)=0$.
In the ordinary case,  we show that at  the expense of adding an even number of intersections between~$D$ and~$F$,  the framing obstruction~$W(D)$ can be  made even.
We apply the stable Whitney trick (Corollary~\ref{cor:StableWhitney}) to remove these intersections and then use Lemma~\ref{lem:modifybyeven} to stabilise~$N$ some more  to obtain~$W(D)=0$.
Thus in both cases we have found an embedded disc~$D$ in a stabilisation of~$N$ with~$D \cap F=\partial D=\gamma$ and~$W(D)=0$ and can therefore perform ambient surgery.

We give some details starting with the characteristic case.
Since the push off~$\gamma'$ is nullhomologous,  we saw that~$F$ represents~$x \oplus 0$ in the stabilisation~$N_r$ (Proposition~\ref{prop:NullhomologousStablyRepresented}) and this class remains characteristic.
Since~$\gamma$ lies in~$N$, one verifies that~$q_{F \subset N}([\gamma])=q_{F \subset N_r}([\gamma])$: both quantities can be calculated using discs in~$N$.
Since~$D \cap F=\partial D=\gamma$, it follows that 
$$0
=q_{F \subset N}([\gamma])
=q_{F \subset N_r}([\gamma])
=W(D) \mod 2.$$
\color{black}
Now apply Lemma~\ref{lem:modifybyeven}, at the cost of a single stabilisation per use,  so that~$W(D)=0$.
We can therefore perform ambient surgery.

We now consider the situation where~$x$ is ordinary.
Since~$D \cap F=\partial D=\gamma$, if~$W(D)$ was even,  then the result would follow by the same argument as in the characteristic case.
We therefore assume that~$W(D)$ is odd; thus~$W(D)+D \cdot F$ is odd.
As noted in Remark~\ref{rem:BoundaryTwist},  we need only find an embedded sphere~$S \subset N$ such that~$Q_N([S],[S])+Q_N^{\partial}([F],[S])$ is odd, for then the disc~$D':=D\# S$ would satisfy
$$W(D')+D' \cdot F
=(W(D)+D \cdot F)+Q_N([S],[S])+Q_N^{\partial}([F],[S])
\equiv 1 + 1 
\equiv0 \mod 2.$$
This would imply that~$W(D')$ and~$D' \cdot F$ have the same parity and therefore boundary twists would allow us to find a new disc~$D''$ with~$W(D'')$ even and~$D'' \cdot F=0$.
The combination of the stable Whitney trick (Corollary~\ref{cor:StableWhitney}) and Lemma~\ref{lem:modifybyeven} would then allow us to arrange for a disc~$D'''$ with~$D''' \cap F=\partial D'''=\gamma$ and~$W(D''')=0$  after stabilisations.
This would allow for ambient surgery.
We have therefore reduced the proof of the remainder of the theorem to the following claim.
\begin{claim}
There is a sphere~$S \subset N$ such that~$Q_N([S],[S])$ and~$Q_N^{\partial}([F],[S])$ have different parities.
\end{claim}
\begin{proof}
We need to find a class~$u \in H_2(N)$ such that~$Q_N(u,u)$ and~$Q_N^\partial(x,u)$ have different parities.
However since we want~$u$ to be representable by an embedded sphere we also arrange for~$u$ to be primitive and ordinary; recall Lemma~\ref{lem:TopologicalWall}.

Since~$H_i(S^3)=0$ for~$i=1,2$,
there is a unique class~$x' \in H_2(N)$ such that~$i_*(x')=x$, where~$i_* \colon H_2(N) \to H_2(N,\partial N)$ denotes the inclusion induced map.
Since, by definition of the intersection form, we have~$Q_N(x',u)=Q_N^\partial(i_*(x'),u)=Q_N^\partial(x,u)$,  our revised goal is to find a primitive,  ordinary~$u \in H_2(N)$ such that~$Q_N(u,u)$ and~$Q_N(x',u)$ have different parities.
The argument for finding this $u$ is identical to the argument in the closed case from~\cite[page 393]{LWCommentarii}.
This concludes the proof of the claim.
\end{proof}

As explained above, the claim ensures we can perform ambient surgery on~$\gamma$.
It remains to argue that that the outcome stably represents~$x$.
We recall why ambient surgery is possible.
The bundle~$\nu(\gamma \subset F)$ is trivial.
Pick one of the two sections(=framings) for this bundle, call it~$s$, and identify~$\nu(\gamma \subset F)$ with~$\gamma \times [-1,1]$.
As~$W(\Delta_i)=0$,  the section~$s$ extends to a section of~$\nu(\Delta \subset N)$.
Since the manifolds are orientable, we obtain a framing of~$\Delta$ and can therefore perform ambient surgery.

Write~$\gamma^\pm:=\gamma \times \lbrace \pm 1\rbrace$.
Since the section~$s$ extends over~$D$,  we obtain disjoint discs~$\Delta^\pm$ with boundary~$\gamma^\pm$.
The ambient surgery removes~$\nu(\gamma \subset F)$ from~$F$,  and adds~$\Delta^\pm$.
Since the surgery is performed in the interior of~$F$ we have~$\partial F'=\partial F$.

We argue that~$[F']=[F]$.
It suffices to show~$F' \cdot z=F \cdot z$ for every~$z \in H_2(N)$.
By construction we have~$z \cdot F'=z \cdot F+z \cdot \Delta^++z \cdot \Delta^-$.
An intersection between~$z$ and~$\Delta$ introduces an intersection between~$z$ and~$\Delta^+$ and also an intersection between~$z$ and~$\Delta^-$, but these intersections have opposite signs because they come at either end of~$\Delta\times[0,1]$.
It follows that~$z \cdot F'=z \cdot F$ for every~$z \in H_2(N)$ and therefore~$[F]=[F'].$ \qedhere
\end{proof}

\section{Sufficient conditions for the existence of a disc}
\label{sec:Sufficient}

Throughout this section, let $N$ be a simply-connected $4$-manifold with boundary~$S^3$, let~$K \subset \partial N \cong S^3$ be a knot with~$H_1(\Sigma_d(K))=0$, let $x \in H_2(N,\partial N)$ be a nonzero class of divisibility~$d$
and assume that a simple disc with boundary~$K$ exists in a stabilisation~$N_k:=N\#^k S^2 \times S^2$, representing $x\oplus0$.

The purpose of this section is to prove Proposition~\ref{prop:TopSplitting}, where we list three conditions that are sufficient to ensure the existence of a simple disc in $N$ representing $x \in H_2(N,\partial N)$. 
The conditions are distilled from an algebraic theorem due to Lee and Wilczy\'nski, applied to our setup.
After recalling the notion of a pointed hermitian form in Section~\ref{sub:PointedHermitian},
Section~\ref{sub:SplittingTheorem} recalls the Lee-Wilczy\'nski theorem;~in essence, their splitting theorem lists five conditions that are sufficient for a pointed hermitian form to split off a hyperbolic form.   
Section~\ref{sub:Fixed} notes that in our setting,  one of their conditions is automatic. 
Section~\ref{sub:Signature} reformulates another of the conditions using Levine-Tristram signature of $K$.
Section~\ref{sub:Sufficient} states the remaining conditions in a way amenable to verification in subsequent sections. We also provide the geometric argument confirming the conditions are sufficient to slice the knot~$K$ in $N$.

\subsection{Pointed hermitian forms}
\label{sub:PointedHermitian}

This section recalls the definition of a pointed hermitian form.

\begin{definition}
A \emph{pointed hermitian form} over a ring with involution $R$ is a triple~$(P,\lambda,z)$ where~$P$ is an $R$-module, $\lambda$ is a hermitian form on $P$, and $z \in P$.
\end{definition}

This notion played a prominent role in~\cite{LWCommentarii,LWKtheory,LWGenus}.

\begin{example}
\label{ex:pointed}
We describe two examples of pointed hermitian forms.
~
\begin{itemize}
\item The \emph{pointed hyperbolic form} over a ring with involution $R$ refers to 
\[
H(R)=\left(R\oplus R,\begin{pmatrix} 0 & 1 \\ 1 & 0 \end{pmatrix}, 0\oplus0\right ).
\]
\item 
Let~$N$,~$K$,~$x$ be as described at the start of the section, and let~$D \subset N \#^k S^2 \times S^2$ be a disc representing~$x \oplus 0$, with boundary~$K$. Consider the~$d$-fold branched cover~$\Sigma_d(D)$ of~$D \subset N_k$.
The branch set, namely the embedded disc \(\widetilde{D} \subset \Sigma_d(D)\), represents a 
class in~\(H_2(\Sigma_d(D), \partial\Sigma_d(D))\).
Since~\(H_1(\Sigma_d(K)) = 0\), the inclusion induces an isomorphism~$H_2(\Sigma_d(D)) \to H_2(\Sigma_d(D),\partial \Sigma_d(D)).$
Denote by \(z \in H_2(\Sigma_d(D)) \) the unique element which is sent to \([\widetilde{D}]\) under this isomorphism.
We also use~$Q_{\Sigma_d(D)}$ to denote the intersection form of~$\Sigma_d(D)$ and consider the hermitian form
\begin{align*}
\lambda_{\Sigma_d(D)} \colon H_2(\Sigma_d(D)) \times H_2(\Sigma_d(D)) &\to \Z[\Z_d] \\
(x,y) \mapsto &=\sum_{g\in \Z_d}Q_{\Sigma_d(D)}(x,gy)g^{-1}.
\end{align*}
The following is a pointed hermitian form over~$\Z[\Z_d]$:
\[
(P,\lambda,z)=(H_2(\Sigma_d(D)),\lambda_{\Sigma_d(D)},z).
\]
\end{itemize}
\end{example}

\subsection{Lee-Wilczy\'nski's splitting theorem}
\label{sub:SplittingTheorem}

We now recall an algebraic ``destabilisation'' result due to Lee-Wilczy\'nski. 
For this, we will need some definitions and notation.

\medbreak

Write $\Z_d=\{1,t,\dots,t^{d-1}\}$ the cyclic group of order $d$ and~$\Lambda:=\Z[\Z_d]$ for its group ring.
We denote by~\(\mathcal{N}:=1+t+\ldots+t^{d-1}\in \Lambda\) the \emph{norm element} and $I=\ker(\Lambda\xrightarrow{\varepsilon}\Z)$ the \emph{augmentation ideal}. Note the free resolution
\begin{equation}
\label{eq:FreeResolution}
\dots\xrightarrow{1-t}\Lambda\xrightarrow{\mathcal{N}}\Lambda\xrightarrow{1-t}\Lambda\xrightarrow{\varepsilon}\Z\to 0.
\end{equation}
Whence, we have $I=\{a\in\Lambda\,|\, \mathcal{N}a=0\}$ the $\mathcal{N}$-torsion elements.
The \emph{Rim square} is the following pullback square, where horizontal maps are the quotient by $(\mathcal{N})$ and vertical maps are the quotient by $I$:
\begin{equation}\label{eq:rim}
\begin{tikzcd}
\Lambda \ar[r] \ar[d]& \Lambda_1\ar[d] \\
\Lambda_0 \ar[r]& \Lambda_2.
\end{tikzcd}
:=
\begin{tikzcd}
\Lambda \ar[r] \ar[d]& \Lambda/(\mathcal{N})\ar[d] \\
\Lambda/I\cong\Z \ar[r]& \Z_d.
\end{tikzcd}
\end{equation}
The isomorphism~$\Lambda/I \cong \Z$ is induced by the augmentation homomorphism.
In what follows, the identification~$\Lambda_0\cong \Z$ will be made implicitly.

\begin{notation}
Let $(P,\lambda, z)$ be a pointed hermitian form over $\Lambda$. We denote by
\begin{itemize}
\item~$T_{\mathcal{N}}P$ the submodule of~$\mathcal{N}$-torsion elements, i.e.\@  those~$x \in P$ such that~$\mathcal{N}x=0$;
\item $(P_0,\lambda_0,z_0)$ the induced $\Z$-valued form on the abelian group~$P_0 :=P/T_{\mathcal{N}}P$;
\item~$\pi_0 \colon P \to P_0$ the quotient map;
\item~$\alpha_0 \colon (P_0,\lambda_0,z_0) \to (P_0',\lambda_0',z_0')$ the morphism induced by a morphism~$\alpha$;
\item~$P^{\Z_d}$ the submodule of $P$ consisting of elements fixed by ${\Z_d}$;
\item~$(P_1,\lambda_1,z_1)$ the induced $\Lambda_1$-valued form on the~$\Lambda_1$-module~$P_1:=P/P^{\Z_d}$;
\item~$\sigma_j(\lambda)$ the signature of the form $\lambda$ restricted to the eigenspace
$$
\left\lbrace x \in P \otimes \C \mid tx=e^{\frac{2\pi ij}{d}}x \right\rbrace.$$
Equivalently,~$\sigma_j(\lambda)$ is the signature of the hermitian form~$(x,y) \mapsto \ev_{e^{2\pi ij/d}}(\lambda(x,y))$, where~$\ev_{e^{2\pi ij/d}} \colon \C[\Z_d] \to \C$ maps the generator of~$\Z_d$ to~$e^{2\pi ij/d}$; see e.g.~\cite[Corollary~3.6]{CimasoniConway}.
\end{itemize}
\end{notation}

\begin{remark}
\label{rem:NonSingular}
One verifies that if $P$ is projective and $\lambda$ is nonsingular, then $\lambda_0$ is also nonsingular.
\end{remark}

Given a pointed hermitian form \((P,\lambda,z)\) over $\Lambda$ where it is known that the corresponding $\Z$-form splits off some number of hyperbolic summands,~$\beta_0\colon (P_0,\lambda_0,z_0)\cong (L, \lambda, x) \oplus H(\Z)^k$, we wish to know when the splitting can be ``lifted'' to the level of the $\Lambda$-form. This desired situation is summarised in the following diagram of quotient maps and isometries:
\[\begin{tikzcd}
	{(P,\lambda,z)} &&& {(P',\lambda',z')\oplus H(\Lambda)^k} \\
	\\
	&&& {(P_0',\lambda_0',z_0')\oplus H(\Z)^k} \\
	\\
	{(P_0,\lambda_0,z_0)} &&& {(L, \lambda, x) \oplus H(\Z)^k.}
	\arrow["\alpha", from=1-1, to=1-4]
	\arrow["\cong"', from=1-1, to=1-4]
	\arrow["{\pi_0}"', from=1-1, to=5-1]
	\arrow["{\pi_0}", from=1-4, to=3-4]
	\arrow["\cong"',"{\alpha_0}", from=5-1, to=3-4, bend left =18]
	\arrow["\cong"', from=5-1, to=5-4]
	\arrow["{\beta_0}", from=5-1, to=5-4]
	\arrow["{\beta_0' \oplus \text{Id}}", from=3-4, to=5-4]
	\arrow["\cong"', from=3-4, to=5-4]
\end{tikzcd}
\]

Lee-Wilczy\'nski's splitting theorem provides the conditions ensuring such a diagram exists and commutes.
For brevity, we say that the hermitian form~$\lambda$ is \emph{weakly even} on a subset $A \subset P$ if, for every~$a \in A$,  it satisfies $\lambda(a,a)=r+\overline{r}$ for some $r \in \Lambda.$

\begin{theorem}[The Splitting Theorem {\cite[Theorem 3.1]{LWGenus}}]\label{thm:splitting}
    Let \((P,\lambda,z)\) be a nonsingular pointed hermitian form over \(\Lambda\).
     Suppose the following conditions hold:
    \begin{enumerate}
            \item \(z\in P^{\Z_d}\);
        \item \(P\) is a projective module of rank \(b+2k\) over \(\Lambda\) with $b,k \geq 1$;
         \item   the $j$-signatures of $\lambda$ satisfy~\(b \geq \max_{0 \leq j<d} \vert \sigma_j(\lambda)\vert\);
        \item there is a nonsingular pointed hermitian \(\Z\)-module \((L,\lambda,x)\) and a splitting \[\beta_0\colon (P_0,\lambda_0,z_0) \cong (L,\lambda,x) \oplus H(\Z)^k;\]
        \item the form $\lambda$ is weakly even on the kernel of
\[\begin{tikzcd}
		{P} & {P_0} & {L\oplus \Z^{2k} } & L.
	\arrow["{\pi_0}", from=1-1, to=1-2]
	\arrow["{\beta_0}", from=1-2, to=1-3]
	\arrow["\proj_1",from=1-3, to=1-4]
\end{tikzcd}\]
    \end{enumerate}
    Then there is a pointed hermitian form \((P', \lambda', z')\) over $\Lambda$ and isometries of pointed hermitian~forms 
        \begin{align*}
            \alpha & \colon (P,\lambda,z) \cong (P',\lambda',z') \oplus H(\Lambda)^k, \\ 
        \beta_0' & \colon (P_0',\lambda_0',z_0') \cong (L, \lambda,x)
        \end{align*}
        such that \(\beta_0 = (\beta_0' \oplus \id)\circ \alpha_0\).
        
        \end{theorem}

\begin{remark}
Theorem~\ref{thm:splitting} is a modified version of the splitting theorem found in \cite[Theorem~3.1]{LWGenus}.
We are restricting to the case of genus zero, which corresponds to setting \(q=0\) in~\cite[Theorem 3.1]{LWGenus}. In this case, as in the other formulation of the splitting theorem from~\cite[Theorem 6.1]{LWKtheory}, there is no $2$-adic condition.
\end{remark}

\subsection{The fixed-point condition}
\label{sub:Fixed}

Let~$N$, $K$, $x$ be as described at the start of the section, and let~$D \subset N \#^k S^2 \times S^2$ be a disc representing $x \oplus 0$ with boundary~$K$.
 Recall from Example~\ref{ex:pointed} that this determines a pointed hermitian form
\[
(P,\lambda,z)=(H_2(\Sigma_d(D)),\lambda_{\Sigma_d(D)},z)
\]
over $\Lambda$, where $z$ corresponds to \([\widetilde{D}]\), the class of the branch set, under the inclusion induced isomorphism \( H_2(\Sigma_d(D))\cong H_2(\Sigma_d(D), \Sigma_d(K))\). 
We now verify $z$ is indeed in the fixed-point set of the $\Z_d$-action.

\begin{lemma}
\label{lem:FixedPointCondition}
Let~$N$ be a simply-connected~\(4\)-manifold with boundary \(S^3\), let~$x \in H_2(N, \partial N)$ be a nonzero homology of divisibility~$d$, and let $K \subset S^3$ be a knot with $H_1(\Sigma_d(K))=0$.
The class~$z$ lies in~\(H_2(\Sigma_d(D))^{\Z_d}\).
\end{lemma}
\begin{proof}
Since $\widetilde{D}$ is the branch set of $\Sigma_d(D)$, the element~$[\widetilde{D}]$ belongs to $H_2(\Sigma_d(D),\Sigma_d(K))^{\Z_d}$. 
Since the inclusion induced isomorphism~$i\colon H_2(\Sigma_d(D)) \to H_2(\Sigma_d(D),\Sigma_d(K))$ is an isomorphism of~\(\Lambda\)-modules, $z=i_*^{-1}([\widetilde{D}]) \in H_2(\Sigma_d(D))^{\Z_d}$, establishing the lemma.
\end{proof}

\subsection{The $j$-signature condition}
\label{sub:Signature}

Next we recast the signature condition of Theorem~\ref{thm:splitting}, in the case where $P=H_2(\Sigma_d(D))$, in terms of Levine-Tristram signatures of the knot $K$.

Let $M$ be a simply-connected $4$-manifold with boundary $S^3$.
Given a properly embedded surface~$F \subset M$ with $H_1(M \setminus F) \cong \Z_d$, the intersection form of~$\Sigma_d(F)$ restricts to a hermitian form on the $e^{2\pi ij/d}$-eigenspace of the $\Z_d$-action
$$\left\{ x \in H_2(\Sigma_d(F);\C) \mid tx=e^{\frac{2\pi ij}{d}} x \right\}.$$
We write~$\sigma_j(\Sigma_d(F))$ for the signature of this restricted form.
Also, we write~$\sigma_K(\omega)$ for the Levine-Tristram signature of a knot~$K$ at~$\omega \in S^1$.
The next lemma, which stems from the work of Rohlin~\cite{RoklinTwoDimensional}, recasts~$\sigma_j(\Sigma_d(D))$ in terms of~$\sigma_K(e^{2\pi ij/d})$ and the algebraic topology of~$M$.

\begin{lemma}
\label{lem:PrelimLT}
Let~$M$ be a simply-connected~$4$-manifold with boundary~$S^3$,  let $K \subset \partial M \cong S^3$ be a knot,  and let $0 \leq j <d$.  If $\Sigma_d(F) \to M$ is a $d$-fold cover with branch set a surface $F \subset M$ with boundary $K$, then
$$\sigma_j(\Sigma_d(F))=\sigma(M)-\frac{2j(d-j)}{d^2} Q_M([F],[F])+\sigma_K(e^{2\pi ij/d}).$$ 
\end{lemma}
\begin{proof}
Consider the closed~$4$-manifold~$X:=M \cup D^4$. 
Let~$G \subset D^4$ be a properly embedded surface with boundary~$K$. A Mayer-Vietoris argument shows the surface~$S:=F \cup G\subset  X$ has~$H_1(X\sm S)\cong~\Z_{d}$. 
The abelianisation~$\pi_1(X\sm S)\to H_1(X\sm S)\cong  \Z_d$ extends~$\pi_1(M\sm F) \to\Z_d$. 
Denote the \(d\)-fold branched cover of~$X$ over \(S\) by \(\Sigma_d(S)\), noting this extends the branched cover of~$M$ over~$F$.
Since~$X$ is closed, a result of Rohlin~\cite{RoklinTwoDimensional} (see also~\cite[Lemma~2.1]{CassonGordonOnSliceKnots}), together with the definition of~$S$ as~$F \cup G$,  implies that
$$ \sigma_j(\Sigma_d(S)) 
 =\sigma(X)-\frac{2j(d-j)}{d^2} Q_X([S],[S])
  =\sigma(M)-\frac{2j(d-j)}{d^2} Q_M([F],[F]).$$ 
On the other hand,  combining Novikov additivity (for~$\Sigma_d(S)=\Sigma_d(F)\cup -\Sigma_d(G)$) with a four-dimensional interpretation of the Levine-Tristram signature~\cite{ViroBranched} (see also~\cite[Theorem~3.1]{ConwaySurvey}) yields
$$ \sigma_j(\Sigma_d(S)) 
=\sigma_j(\Sigma_d(F)) -\sigma_j(\Sigma_d(G)) 
=\sigma_j(\Sigma_d(F)) -\sigma_K(e^{2\pi i j/d}).  
$$ 
The lemma now follows by combining these two results.
\end{proof}

Finally,  we note that the signature condition in Theorem~\ref{thm:Main} is necessary.
\begin{proposition}
\label{prop:SignatureCondition}
For a simply-connected~$4$-manifold~$M$ with boundary~$S^3$,   a~$\Z_d$-surface~$F \subset M$ with boundary a knot~$K \subset S^3$,  and $0 \leq j<d$, we have
$$ b_1(F)+b_2(M) \geq \Bigg| \sigma(M)-\frac{2j(d-j)}{d^2} Q_M([F],[F]) +\sigma_K(e^{2\pi ij/d}) \Bigg|=\Bigg|\sigma_j(\Sigma_d(F))\Bigg|.$$ 
\end{proposition}
\begin{proof}
The first inequality is due (in much greater generality) to Gilmer~\cite[Remark following Corollary 4.2]{GilmerConfiguration}.
Gilmer's result is stated for $d$ a prime power dividing $x$ but without the assumption that $F$ is simple. 
A little work shows that the prime power assumption can be dropped if $F$ is a~$\Z_d$-surface (so, in particular, $d$ is the divisibility of $x$).
The equality is Lemma~\ref{lem:PrelimLT}.

For the reader's convenience, we outline why, when $F$ is simple, Gilmer's result does not require the prime power hypothesis.
For $j=0$,  the result already follows from Gilmer's work. 
Now assume that $j>0$, set $\omega:=e^{2\pi ij/d}$,  and for~$X \in \{ X_F^d,\Sigma_d(F)\}$
$$ H_2(X;\C)_\omega:=\left\{ x \in H_2(X;\C) \mid tx=\omega x \right\}.$$
Since $|\sigma_j(\Sigma_d(F))| \leq \dim_\C H_2(\Sigma_d(F))_\omega$, it suffices to prove that~$\dim_\C H_2(\Sigma_d(F))_\omega \leq b_1(F)+b_2(M)$.
In fact we will prove equality.
We use~$\C^\omega$ to denote both the $\Lambda_\C:=\C[\Z_d]$-module and the $\Lambda$-module obtained by linearly extending~$tx:=\omega x$ for $x \in \C$.
For~$X \in \{ X_F^d,\Sigma_d(F)\}$, we have 
\begin{equation}
\label{eq:EigenspaceTensor}
H_2(X)_\omega 
\cong H_2(X;\C) \otimes_{\Lambda_\C} \C^\omega 
\cong H_2(X) \otimes_{\Lambda} \C^\omega,
\end{equation}
where the first equality follows e.g.~from~\cite[Proposition 3.3]{CimasoniConway} and the second from the associativity of the tensor product.

Since $F$ is simple, the spaces $X_F^d$ and $\Sigma_d(F)$ are simply-connected.
Recalling that $\Lambda_0=\Lambda/I$ denotes $\Z$ with the $\Lambda$-module structure induced by augmentation,  and since $\Z_d$ acts trivially on the homology of~$\overline{\nu}(\widetilde{F}) \cong \widetilde{F} \times D^2$,  
the Mayer-Vietoris exact sequence for $\Sigma_d(F)$ yields
$$  0 \to \overbrace{H_2(\widetilde{F} \times S^1)}^{\cong \Lambda_0^{2g}} \to H_2(X_F^d) \to H_2(\Sigma_d(F)) \xrightarrow{} \Lambda_0  \to 0.$$
Applying~$-\otimes_{\Lambda} \C^\omega$ and using that~$\Lambda_0 \otimes_{\Lambda} \C^\omega  =0$ as well $\Tor_1^{\Z[\Z_d]}(\Lambda_0,\C^\omega)=0$,  we deduce from~\eqref{eq:EigenspaceTensor} that the inclusion induces an isomorphism
$$ H_2(X^d_F)_\omega \xrightarrow{\cong} H_2(\Sigma_2(F))_\omega. $$ 
From now on,  we use the isomorphism~$H_2(X^d_F)_\omega \cong H_2(X_F;\C^\omega)$ with twisted homology; see e.g.~\cite[Proof of Lemma~3.6]{ConwayNagel}.
It thus remains to show that~$\dim_\C H_2(X_F;\C^\omega) = b_1(F)+b_2(M)$.
First, a calculation, using that $F$ is simple, shows that $H_1(X_F;\C^\omega)=0$.
Setting $b_i^\omega(X_F):=H_i(X_F;\C^\omega)$ and~$\chi^\omega(X_F)=\sum_i(-1)^i b_i^\omega(X_F)$,  the outcome now follows from an Euler characteristic calculation:
$$b_2^\omega(X_F)=\chi^\omega(X_F)=\chi(X_F)=(\chi(M)-1)+2g=b_2(M)+b_1(F).$$
We have thus obtained $|\sigma_j(\Sigma_d(F))| \leq \dim_\C H_2(\Sigma_d(F))_\omega=b_2(X_F)^\omega=b_1(F)+b_2(M)$, as required.
\end{proof}

\begin{remark}
Gilmer's formula involves~$-\sigma_K(e^{2\pi ij/d})$ instead of~$+\sigma_K(e^{2\pi ij/d})$.
We believe this is a minor error,  as can be verified by recalling that the left handed trefoil bounds a~$\Z_2$-disc in~$\C P^2$; recall Example~\ref{ex:CP2d=2} and Remark~\ref{rem:Trefoil}.
This was already noted implicitly in~\cite[Theorem~3.1]{YasuharaConnecting}, as well as in~\cite[Theorem 3.6]{ManolescuMarengonPiccirillo} and~\cite[Lemma 2.1]{MarengonMillerRayStipsicz}: each of these results cites Gilmer but with~$+\sigma_K(e^{2\pi ij/d})$ in place of~$-\sigma_K(e^{2\pi ij/d})$.
The appearance of this sign can be traced back to the proof of Lemma~\ref{lem:PrelimLT} and, more specifically to the decomposition~$\Sigma_d(S)=\Sigma_d(F)\cup -\Sigma_d(G)$.
\end{remark}

\subsection{The sufficient conditions}
\label{sub:Sufficient}

In this section, we distill Lee-Wilczy\'nski's splitting theorem, to list sufficient conditions for lifting a hyperbolic summand from the $\Z$-form to the $\Lambda$-form, in our case of interest. 
For completeness, we provide the reader with the details to finish
off the surgery argument (described in a slightly different
form by Lee-Wilczy\'nski's~\cite[Proposition 4.1 and page~410]{LWCommentarii} in the spherical case)
for finding a simple slice disc in the destabilised manifold $N$, given a simple slice disc in the stabilised manifold $N_k$.

\begin{proposition}
\label{prop:TopSplitting}
Let~$N$ be a simply-connected \(4\)-manifold with boundary \(S^3\), let~$x \in H_2(N, \partial N)$ be a nonzero class of divisibility~$d$, and let $K \subset S^3$ with $H_1(\Sigma_d(K))=0$ and 
\begin{equation}
\label{eq:SignatureHyp}
b_2(N)\geq \max_{0\leq j<d}\Bigg|\sigma (N)-\frac{2j(d-j)}{d^2} Q_N(x,x)+\sigma_K (e^{\frac{2\pi i j}{d}})\Bigg|.
\end{equation}
Let~$D \subset N \#^k S^2\times S^2$ be a simple disc representing~$x \oplus 0 \in H_2(N, \partial N) \oplus H_2(\#^k S^2\times S^2)$ with boundary $K$.
Assume the following conditions hold:
\begin{enumerate}
\item Freeness condition: the $\Lambda$-module $H_2(\Sigma_d(D))$ is free of rank $b_2(N)+2k.$
\item Splitting the form over $\Z$: there is an isomorphism
 \[\beta_0\colon (H_2(\Sigma_d(D))_0,(\lambda_{\Sigma_d(D)})_0,z_0) \cong (H_2(N),Q_N,i_*^{-1}(x)) \oplus H(\Z)^k,\]
 where $i_* \colon H_2(N) \to H_2(N,\partial N)$ is the inclusion induced map.
 \item  The evenness condition: the form $\lambda_{\Sigma_d(F)}$ is weakly even on the kernel of
\[\begin{tikzcd}
	{\eta \colon H_2(\Sigma_d(D))} & {H_2(\Sigma_d(D))_0} & {H_2(N)\oplus \Z^{2k} } & H_2(N).
	\arrow["{\pi_0}", from=1-1, to=1-2]
	\arrow["{\beta_0}", from=1-2, to=1-3]
	\arrow["{\proj_1}",from=1-3, to=1-4]
\end{tikzcd}\]
\end{enumerate}
Then $x$ is represented by a simple disc in $N$ with boundary~$K$.
\end{proposition}
\begin{proof}
The strategy of proof is to perform surgery on $N_k$ to destabilise back to $N$, and do so away from $D$, so that the stable slice disc is preserved in the destabilised manifold. To do this, we will use the splitting theorem (Theorem~\ref{thm:splitting}) to obtain a hyperbolic summand $H(\Lambda)^{\oplus k}$ for the pointed~$\Lambda$-form of $\Sigma_d(D)$. 
We will then show this summand lies in the image of~$H_2(X_D^d) \to H_2(\Sigma_d(D))$,  where~$X_D^d$ denotes the $d$-fold (unbranched) cover of the disc exterior~$X_D:=N_k \setminus \nu(D)$. 
As the disc~$D$ is simple, $X_D^d$ is the universal cover of~$X_D$, so the Hurewicz map is an isomorphism~$\pi_2(X_D)\cong H_2(X_D)$. 
We will thus be able to represent generators for $H(\Lambda)^{\oplus k}$ spherically in~$X_D$, for the desired surgery on~$N_k$, occuring away from $D$.

We begin by verifying that the splitting theorem (Theorem~\ref{thm:splitting}) can be applied. 
First, the assumption that $H_1(\Sigma_d(K))=0$ ensures that the $\Z$-valued intersection form~$Q_{\Sigma_d(D)}$ is nonsingular and some algebra shows that~$\lambda_{\Sigma_d(D)}$ is therefore also nonsingular.
Next, Lemma~\ref{lem:FixedPointCondition} ensures that~$z:=i_*^{-1}([\widetilde{D}]) \in H_2(\Sigma_d(D))$ is fixed by the $\Z_d$ action, confirming the first hypothesis of the splitting theorem. Proposition~\ref{prop:SignatureCondition} shows that~\eqref{eq:SignatureHyp} is a reformulation of the third assumption of the splitting theorem.
The remaining assumptions of the present theorem ensure that we can apply the splitting theorem.

The splitting theorem provides isometries~$\alpha$ and $\beta_0'$ that make the following diagram commute
\[
\begin{tikzcd}
	{(H_2(\Sigma_d(D)),\lambda_{\Sigma_d(D)},z)} &&& {(\Lambda^{b_2(N)},\lambda',z') \oplus (H(\Lambda),0)^{\oplus k} } \\
	\\
	&&& {(\Z^{b_2(N)},\lambda_0',z_0') \oplus (H(\Z),0)^{\oplus k} } \\
	\\
	{(H_2(\Sigma_d(D))_0,(\lambda_{\Sigma_d(D)})_0,z_0)} &&& {(H_2(N),Q_N,i_*^{-1}(x)) \oplus (H(\Z),0)^{\oplus k}.}
	\arrow["\alpha", from=1-1, to=1-4]
	\arrow["\cong"', from=1-1, to=1-4]
	\arrow["{\pi_0}"', from=1-1, to=5-1]
	\arrow["{\pi_0}", from=1-4, to=3-4]
	\arrow["\cong"',"{\alpha_0}", from=5-1, to=3-4, bend left =14]
	\arrow["\cong"', from=5-1, to=5-4]
	\arrow["{\beta_0}", from=5-1, to=5-4]
	\arrow["{\beta_0' \oplus \text{Id}}", from=3-4, to=5-4]
	\arrow["\cong"', from=3-4, to=5-4]
\end{tikzcd}
\]
In particular, the pointed hermitian form on~$H_2(\Sigma_d(D))$ splits as
\begin{equation}
\label{eq:PointedSplitting}
(H_2(\Sigma_d(D)),\lambda_{\Sigma_d(D)},z) 
\cong 
(\Lambda^{b_2(N)},\lambda',z') \oplus (H(\Lambda)^{\oplus k},0).
\end{equation}

We now show the $H(\Lambda)^{\oplus k}$ summand lies in the image of~$H_2(X_D^d) \to H_2(\Sigma_d(D))$. Consider the decomposition given by \(\Sigma_d(D)= X_D^d \cup \overline{\nu}(\widetilde{D})\). 
Since $\overline{\nu}(\widetilde{D}) \cong D^2 \times D^2$, the gluing occurs along the trivial circle bundle~$\partial \overline{\nu}(\widetilde{D}) \setminus \nu(K) \cong \widetilde{D} \times S^1$.
Consider the following section of the resulting Mayer-Vietoris sequence:
\begin{equation}
\label{eq:SESBranched}
0 \to H_2(X_D^d) \xrightarrow{j} H_2(\Sigma_d(D)) \xrightarrow{\partial} \overbrace{H_1(\widetilde{D} \times S^1)}^{\cong\Z} \to 0.
\end{equation}
Since~$j$ preserves the $\Z$-valued intersection forms, it preserves the $\Lambda$-valued intersections forms. 
It follows that $j$ induces an isometry
$$(H_2(X_D^d),\lambda_{X_D^d}) \cong (\im(j),\lambda_{\Sigma_d(D)}|_{\im(j)})=(\ker(\partial),\lambda_{\Sigma_d(D)}|_{\ker(\partial)}).$$
Arguing similarly to e.g.~\cite[Lemma 5.1]{ConwayPowell}  the connecting homomorphism in the Mayer-Vietoris sequence in~\eqref{eq:SESBranched} satisfies~\(\partial(a) =Q_{\Sigma_d(D)}^\partial(a,[\widetilde{D}])= Q_{\Sigma_d(D)}(a,z)\) for all \(a\in H_2(\Sigma_d(D))\).
It follows~that 
$$ \ker(\partial)=\{ y \in H_2(\Sigma_d(D)) \mid Q_{\Sigma_d(D)}(y,z)=0\}.$$
For brevity, we now use $\alpha$ to identify~$H_2(\Sigma_d(D)) =\Lambda^b \oplus \Lambda^{2k}$.
A short verification shows that~$\ker(\partial)=(\ker(\partial) \cap \Lambda^b) \oplus \Lambda^{2k}$, and
%
it follows that we obtain the direct sum decomposition
$$(H_2(X_D^d),\lambda_{X^d_D})
 \cong (\ker(\partial),\lambda_{\Sigma_d(D)}|_{\ker(\partial)})
= (\ker(\partial) \cap \Lambda^b, \lambda_{\Sigma_d(D)}|_{\ker(\partial)}) \oplus (H(\Lambda),0)^{\oplus k}.
 $$
Since $\Z_d$ is a good group, Freedman's Sphere Embedding Theorem~\cite{Freedman} (see \cite[p.~292]{DETBook} for a statement) implies we can represent the hyperbolic summand of $\pi_2(X_D) \cong H_2(X_D^d)$ by~$2k$ framed embedded spheres in $X_D$.
Perform surgery on these spheres in order to obtain a~$4$-manifold~$X'$ with~$\partial X'=\partial X_D$ and~$\pi_1(X')=\pi_1(X_D)$,
and so that the branched cover of~$N':=X' \cup (D\times D^2)$ along~$D':=D \times \{0\}$ has pointed equivariant intersection form isometric (via~$\alpha)$ to~$\lambda'$.
Here, slightly abusing notation, we will write~$z'$ both for the distinguished element of~$(\Lambda^{b_2(N)},\lambda',z')$ and for the homology class of the class determined by the lift of~$D'$ to~$\Sigma_d(D')$.

Set~$x':=[D'] \in H_2(N',\partial N')$ and write~$i' \colon N' \to (N',\partial N')$ for the inclusion map.
 Proposition~\ref{prop:SplitDownstairs} implies that the projection~$\Sigma_d(D') \to N'$ induces a pointed isometry~$(\lambda_{\Sigma'_d(D')})_0 \cong Q_{N'}$ leading to the isometries
\begin{align*}
(H_2(N'),Q_{N'},(i')_*^{-1}(x')) 
&\cong (H_2(\Sigma_d(D'))_0,(\lambda_{\Sigma'_d(D')})_0,z')
\xrightarrow{\alpha_0,\cong} (\Z^{b_2(N)},\lambda_0',z_0') \\
&\xrightarrow{\beta_0',\cong}  (H_2(N),Q_N,i_*^{-1}(x)). 
\end{align*}
By additivity of the Kirby-Siebenmann invariant~\cite{FreedmanQuinn} (see~\cite[Theorem~9.2]{FriedlNagelOrsonPowell} for a statement), we also obtain~$\ks(N')=\ks(X')=\ks(X_D)=\ks(N).$
Since $\partial N'=\partial N \cong S^3$, we can thus use work of Freedman and Quinn
(\cite[Theorem~1.5]{Freedman},~\cite[Chapter~10]{FreedmanQuinn})
 to realise this isometry by a homeomorphism~$N' \cong N$ rel. boundary that takes~$(i'_*)^{-1}([D'])$ to~$(i_*)^{-1}([D])$
and thus~$[D']$ to~$[D]=x$.
The image of~$D'$ by this homeomorphism yields a disc in~$N$ representing~$x$ with boundary~$K$.
\end{proof}

\section{The freeness condition}
\label{sec:H2Free}

The next three sections are devoted to establishing the three conditions listed in Proposition~\ref{prop:TopSplitting}.
We fix some notation that we will frequently refer to.
\begin{notation}
\label{not:TheUsual}
Let $N$ be a simply-connected $4$-manifold with boundary $S^3$,  set~$N_k:=N \#^k S^2 \times S^2$, and let~$K \subset \partial N \cong S^3$ be a knot with~$H_1(\Sigma_d(K))=0$.
Let $x \in H_2(N,\partial N)$
be a nonzero class of divisibility $d$, and let~$D \subset N_k$ be a disc with boundary~$K$ that represents the class~$x \oplus 0$. 
We write~\(\Sigma_d(D)\) for the \(d\)-fold branched cover of \(N_k\) over \(D\), as well as~$X_D$ for the exterior of~\(D\) in~\(N_k\), and~$X_D^d$ for the unbranched (universal)~\(d\)-fold cover of $X_D$.
\end{notation}

The goal of this section is to establish the first condition from Proposition~\ref{prop:TopSplitting}:~$H_2(\Sigma_d(D))$ is~$\Lambda$-free of rank $b_2(N)+2k$ (Proposition~\ref{prop:H2Free}).
After recalling some facts concerning modules over $\Lambda$ (Section~\ref{sub:AlgebraZZd}),  the proof has two main steps:
\begin{enumerate}
\item Prove that \(H_2(\Sigma_d(D))\) is projective (Section~\ref{sub:H2Proj}).
\item Prove that  \(H_2(\Sigma_d(D))\) is stably free (Section~\ref{sub:H2StablyFree}).
\end{enumerate}
The conclusion that~$H_2(\Sigma_d(D)) \cong \Lambda^{b_2(N)+2k}$ then follows promptly.
Our overall strategy follows Lee and Wilczy\'nski's approach in the closed case~\cite{LWCommentarii,LWKtheory}.

\subsection{Modules over $\Lambda$}
\label{sub:AlgebraZZd}

We record some facts concerning projective and stably free modules over group rings of finite cyclic groups.
In what follows,  modules are understood to be finitely generated.

\medbreak

Recall that free modules are stably free and stably free modules are projective.
The converses are in general false, but over $\Lambda$ there are some partial converses which we now collect.

\begin{theorem}
\label{thm:ModulesOverZZd}
~
\begin{enumerate}
\item Stably free $\Lambda$-modules are free.
\item A projective $\Lambda$-module $P$ is stably free if $P \otimes_{\Lambda} \Lambda_1$ is stably free.
\end{enumerate}
\end{theorem}
\begin{proof}
The first assertion is due to Jacobinski~\cite{Jacobinski}; we also refer to Johnson~\cite{Johnson} for clearer statements.
The second statement is implicit in~\cite[proof of Theorem 3.4]{LWCommentarii} but we give some details for the reader's convenience.

We begin by detailing Lee and Wilczy\'nski's argument that~$\proj_1 \colon \Lambda \to \Lambda_1$ induces an injection~$( \proj_1)_* \colon K_0(\Lambda) \to K_0(\Lambda_1)$; the result is well known when $d$ is a prime power (see e.g.~\cite[page 29]{MilnorKTheory}) and is stated in~\cite[page~400]{LWCommentarii} for arbitrary $d$, but we provide some additional references concerning some steps of Lee and Wilczy\'nski's proof.

The Mayer-Vietoris sequence associated to the Rim square in~\eqref{eq:rim} takes the form
 \[\begin{tikzcd}
	{K_1(\Lambda_1) \oplus K_1(\Z)} & {K_1(\Z_d)} & {K_0(\Lambda)} && {K_0(\Lambda_1) \oplus K_0(\Z)} & {K_0(\Z_d)\to 0.}
	\arrow["{f_1}", from=1-1, to=1-2]
	\arrow["\partial", from=1-2, to=1-3]
	\arrow["\begin{array}{c} \begin{array}{c} \bsm {(\proj_1)_*} \\ (\proj_0)_* \esm \end{array} \end{array}", from=1-3, to=1-5]
	\arrow["{f_0}", from=1-5, to=1-6]
\end{tikzcd}\]
The same argument as in~\cite[page 345 (c) and 349-350]{LamSiu} shows that $\eta_1:=f_1|_{K_1(\Lambda_1)}$ is surjective; here note that~\cite{LamSiu} assumes that $d$ is a prime, but the argument holds without this assumption.
It follows that~$\partial$ is the zero map leading to the short exact sequence
\[\begin{tikzcd}
	0 & {K_0(\Lambda)} && {K_0(\Lambda_1) \oplus K_0(\Z)} & {K_0(\Z_d)} & 0.
	\arrow[from=1-1, to=1-2]
	\arrow["\begin{array}{c} \begin{array}{c} \bsm {(\proj_1)_*} \\( \proj_0)_* \esm \end{array} \end{array}", from=1-2, to=1-4]
	\arrow["{f_0}", from=1-4, to=1-5]
	\arrow[from=1-5, to=1-6]
\end{tikzcd}\]
Since $\eta_0 =f_0|_{K_0(\Z)}\colon K_0(\Z) \to K_0(\Z_d)$ is injective~\cite[Example 1.5.10(a)]{Rosenberg},  the exactness of this sequence ensures that~\((\proj_1)_*\) is also injective, as required.

We now explain the argument, implicit in~\cite{LWCommentarii}, that the injectivity of~$K_0(\Lambda) \to K_0(\Lambda_1)$ suffices to conclude.
For this, we recall that the \emph{reduced projective class group}~$\widetilde{K}_0(R)$ of a commutative ring $R$ is the kernel of the rank map~$K_0(R)\to H_0(R)$~\cite[Definition~II.2.3]{WeibelKBook}. 
The inclusion~$H_0(R)\subset K_0(R)$ splits the rank map, so we obtain an isomorphism~$K_0(R)\cong H_0(R)\oplus \widetilde{K}_0(R)$. 
It follows that stably free modules are trivial in $\widetilde{K}_0(R)$ and,  for $R=\Z[G]$ with $G$ finite,  
the converse is also true: this uses a theorem of Swan (see~\cite[Example~II.2.4]{WeibelKBook}),  according to which, when~$G$ is finite, there is an alternative description of~$\widetilde{K}_0(\Z[G])$ as the quotient~$\widetilde{K}_0(\Z[G])=K_0(\Z[G])/K_0(\Z)$, and in this case also~$H_0(R)\cong\Z$.

The conclusion now follows.
Indeed, the injection~$(\proj_1)_*\colon  K_0(\Lambda) \to K_0(\Lambda_1)$ induces an injection~$(\proj_1)_*\colon \widetilde{K}_0(\Lambda) \to \widetilde{K}_0(\Lambda_1)$, 
the assumption that $P \otimes_{\Lambda} \Lambda_1$ is stably free over $\Lambda_1$ implies that~$(\proj_1)_*([P])=0 \in \widetilde{K}_0(\Lambda_1)$,  so $[P]=0 \in \widetilde{K}_0(\Lambda)$ and thus~$P$ is stably free.
\end{proof}

\subsection{The module  \(H_2(\Sigma_d(D))\) is projective.}\label{sub:H2Proj}

We carry out the first step towards proving~$H_2(\Sigma_d(D))$ is~$\Lambda$-free of rank $b_2(N)+2k$ (Proposition~\ref{prop:H2Free}). 

\begin{proposition}
\label{prop:SphereStablyFree}
    The \(\Lambda\)-module \(H_2(\Sigma_d(D))\) is projective.
\end{proposition}
\begin{proof}
Consider the double~$(V,S):=(N,D) \cup_{\id} (N,D)$, where~$V$ is now a closed simply-connected~$4$-manifold containing the sphere~$S$.
By van Kampen, the exterior~$X_S=X_D \cup_{X_K} X_D$ of~$S$ has~$\pi_1(X_S) \cong \Z_d$, allowing us to consider the (simply-connected) \(d\)-fold branched cover~$\Sigma_d(S)$.
Here, $X_K$ denotes the exterior of $K$.
Observe that  \(\Sigma_d(S) = \Sigma_d(D) \cup_{\Sigma_d(K)} \Sigma_d(D).\) The inclusion of one side of this double determines a pair~\((\Sigma_d(S),\Sigma_d(D))\). Consider the following portion of the corresponding long exact sequence:
\[\begin{tikzcd}
	0 & {H_2(\Sigma_d(D))} & {H_2(\Sigma_d(S))} & {H_2(\Sigma_d(S),\Sigma_d(D))} & 0.
	\arrow[from=1-1, to=1-2]
	\arrow[from=1-2, to=1-3, "\iota_*"]
	\arrow[from=1-3, to=1-4]
	\arrow[from=1-4, to=1-5]
\end{tikzcd}\]
Since \(\Sigma_d(S)\) is the double of \(\Sigma_d(D)\), we can define a ``folding'' map \(s\colon \Sigma_d(S) \to \Sigma_d(D)\) which maps a point in \(\Sigma_d(S)\) to its corresponding point in \(\Sigma_d(D)\). It is straightforward to check that~$s_*$ is a splitting,~\(s_*\circ\iota_* = \id_{\Sigma_d(D)}\), and so the sequence splits:
$$  H_2(\Sigma_d(S)) \cong  H_2(\Sigma_d(D))\oplus H_2(\Sigma_d(S),\Sigma_d(D)). $$
In \cite[page 399]{LWCommentarii}, it is shown that \(H_2(\Sigma_d(S))\) is projective, and thus so is~\(H_2(\Sigma_d(D))\).
\end{proof}

\begin{remark}
The proof in~\cite[page 399]{LWCommentarii} relies on several facts from group cohomology which are not referenced. 
We give some indications for the interested reader.
First, given a simple sphere~$S \subset X$ in a closed simply-connected~$4$-manifold~$X$ representing a divisibility $d$ homology class~$x \in H_2(X)$, since~$H_2(\Sigma_d(S))$ is free abelian, it suffices to prove that~$H_2(\Sigma_d(S))$ is cohomologically trivial~\cite[Theorem 8.10]{BrownCohomology}.
This further reduces to proving that~$H^i(\Z_{p^r};H_2(\Sigma_d(S)))=0$ for~$p$ prime and~$i=2,3$ thanks to~\cite[Theorem~8.9]{BrownCohomology} (and the fact that Tate homology agrees with usual group homology in positive degrees~\cite[page~134]{BrownCohomology}).
We also note that for~$I_d$ the augmentation ideal, the equality~$H^{i+1}(\Z_{p^r};I_d)=H^i(\Z_{p^r};\Z)$ on~\cite[page 399]{LWCommentarii} uses that~$\Z[\Z_d]$ is a  free~$\Z[\Z_{p^r}]$-module (see~\cite[pages 13-14]{BrownCohomology}) to deduce that~$H^i(\Z_{p^r};\Z[\Z_d])=0$ for~$i>0$.
\end{remark}

\subsection{The module  \(H_2(\Sigma_d(D))\) is stably free.}
\label{sub:H2StablyFree}

This section proves that \(H_2(\Sigma_d(D))\) is stably free.
First however,  we show that $\pi_2(X_D) \cong H_2(X_D^d)$ is stably isomorphic to the augmentation ideal~$I=\ker(\Lambda\xrightarrow{\varepsilon}\Z)$.
The argument originated from discussions with Daniel Kasprowski and relies on the following observation.

\begin{lemma}
\label{lem:2Complex}
 If a smooth $4$-manifold $X$ has nonempty boundary and~$\pi_1(\partial X) \to \pi_1(X)$ is surjective, then~$X$ has the homotopy type of a $2$-complex.
\end{lemma}
\begin{proof}
The manifold $X$ can be obtained from $\partial X$ by attaching $1$-,$2$-,$3$-, and $4$-handles.
Since the inclusion induced map~$\pi_1(\partial X) \to \pi_1(X)$ is surjective,  one can trade $1$-handles for $3$-handles.
Turning the handle decomposition upside down,  it follows that $X$ admits a handle decomposition with only~$0$-,$1$-, and $2$-handles so the conclusion follows.
\end{proof}

Next,  as in~\cite[page 398]{LWCommentarii}, we analyze $\pi_2(X_D) \cong H_2(X_D^d)$.
\begin{lemma}
\label{lem:StablyAugmentation}
    The \(\Lambda\)-module \(H_2(X_D^d)\) is stably isomorphic to the augmentation ideal~\(I\).
\end{lemma}
\begin{proof}
As $X_D$ has nonempty boundary, it has the homotopy type of a $3$-complex see e.g.~\cite[Theorem 3.16]{FriedlNagelOrsonPowell}.
Consider the singular $\Lambda$-chain complex of this $3$-complex:
$$ 0 \to C_3 \xrightarrow{ d_3} C_2 \xrightarrow{ d_2} C_1 \to C_0 \to 0.$$
As explained in~\cite[Proof of Proposition 2.2]{WilczynskiPeriodic}, 
this leads to a short exact sequence
$$0 \to \ker( d_2) \to H_2(X_D^d) \oplus C_2 \to \coker( d_3) \to 0.$$

We claim that~$\ker( d_2)$ is stably isomorphic to the augmentation ideal $I$.
To see this, write~$\mathcal{N} \in \Lambda$ for the norm element,  and consider the free $ \Lambda$-resolution of $\Z$ given by 
$$\ldots \xrightarrow{t-1} \Lambda \xrightarrow{\mathcal{N}}  \Lambda \xrightarrow{t-1} \Lambda \xrightarrow{\operatorname{aug}} \Z \to 0.$$
In particular note that $I=\im(t-1)=\ker(\mathcal{N})$, leading to the following two exact sequences:
\begin{align*}
0 \to  \ker(d_2) \to C_2 \xrightarrow{d_2} C_1 \xrightarrow{d_1} C_0 \xrightarrow{\operatorname{proj}} \Z \to 0 \\
0 \to I \to \Lambda \xrightarrow{\mathcal{N}}  \Lambda \xrightarrow{t-1} \Lambda \to \Z \to 0.
\end{align*}
Since the $C_i$ are free, Schanuel's lemma (see e.g.~\cite[Chapter VIII, Lemma 4.4]{BrownCohomology}) now implies that $\ker( d_2)$ is stably isomorphic to $I$.

The lemma will now be proved if we can show~$H_2(X_D^d) \cong H_2(X_D;\Lambda)$ is stably isomorphic to~$\ker( d_2)$.
In the case that $X_D$ is smooth, Lemma~\ref{lem:2Complex} implies that~$X_D$ has the homotopy type of a~$2$-complex. Thus in this case,~$d_3=0$ and~$C_2(X_D;\Lambda)=\coker(d_3)$ so that~$\ker(d_2)=H_2(X_D;\Lambda)$ and the lemma is proved.

If $X_D$ is not smooth, results of Freedman-Quinn \cite{FreedmanQuinn} (see~\cite[Theorem~9.9]{FriedlNagelOrsonPowell} for a statement), imply that it can be made smoothable by taking connected sums with copies of~$S^2 \times S^2$ and a copy of Freedman's~$E_8$ manifold if~$\operatorname{ks}(X_D)=1$. 
Call the resulting smoothable manifold $X'$. 
The addition of these connected summands does not change the fundamental group and so we still have~$H_0(X';\Lambda) \cong \Z$ and~$H_1(X';\Lambda)=0$. The addition of the connected summands adds free summands to $H_2(X_D;\Lambda)$, meaning there is a stable isomorphism $H_2(X_D;\Lambda)\cong_{s}H_2(X';\Lambda)$.

We compare our resolution of $\Z$ from above to the resolution of~$H_0(X';\Lambda) \cong \Z$ obtained by considering the chain complex of the universal cover of $X'$:
\begin{align*}
0 \to  \ker(d_2) \to C_2 \xrightarrow{d_2} C_1 \xrightarrow{d_1} C_0 \xrightarrow{\operatorname{proj}} \Z \to 0\\
0 \to \ker(d_2') \to C_2' \xrightarrow{d_2'} C_1' \xrightarrow{d_1'} C_0' \xrightarrow{\operatorname{proj}} \Z \to 0.
\end{align*}
Since the~$C_i$ and~$C_i'$ are free (and therefore projective),  Schanuel's lemma implies that
$$ \ker(d_2) \oplus C_2' \oplus C_1 \oplus C_0' \cong  \ker(d_2') \oplus C_2 \oplus C_1' \oplus C_0.$$
Since the~$C_i$ and~$C_i'$ are free, it follows that we have a stable isomorphism~$\ker(d_2)\cong_s \ker(d_2')$. On the other hand, as $X'$ is smoothable, the argument from above that used Lemma~\ref{lem:2Complex} shows that~$\ker(d'_2)=H_2(X';\Lambda)$,
and thus
$H_2(X_D;\Lambda)\cong_{s} H_2(X';\Lambda)=\ker(d_2')\cong_{s}\ker(d_2)\cong_s  I.
$
\end{proof}

\begin{remark}
\label{rem:LWStablyAugmentation}
We comment on Lee and Wilczynski's argument
 that for a simple sphere~$S \subset X$ in a closed manifold,~$\pi_2(X_S)\cong_s I$~\cite[page 398]{LWCommentarii}.
Their argument relies on prior work of Wilczynski~\cite[Proposition 2.3]{WilczynskiPeriodic} which makes use of the loop and suspension functors for~$\Lambda$-modules described in \cite[Page 515]{WallPeriodicProjective} and \cite{Heller} 
 and, implicitly, on certain of their 
exactness properties (note that the effect of these functors is only defined up to the addition with projectives).
Thus while it is plausible such an argument may carry through
to discs,  we believe that that the inclusion of a less sophisticated 
proof also has some appeal.
\end{remark}
\color{black}

We now prove most of the the main result of this section.

\begin{proposition}
\label{prop:SphereStablyFree}
    The \(\Lambda\)-module \(H_2(\Sigma_d(D))\) is stably free.
\end{proposition}
\begin{proof}
We know from Section~\ref{sub:H2Proj} that~$H_2(\Sigma_d(D))$ is $\Lambda$-projective.
By Theorem~\ref{thm:ModulesOverZZd},  in order to show that~$H_2(\Sigma_d(D))$ is $\Lambda$-stably free, it suffices to prove that~$H_2(\Sigma_d(D)) \otimes_{\Lambda} \Lambda_1$ is~$\Lambda_1$-stably free.
Since $H_1(\Sigma_d(K))=0$,  the exact sequence of the pair and duality yield 
$$ H_2(\Sigma_d(D) \cong H_2(\Sigma_d(D),\partial \Sigma_d(D)) \cong  H^2(\Sigma_d(D)).$$
We will therefore show that $H^2(\Sigma_d(D)) \otimes_\Lambda \Lambda_1$ is~$\Lambda_1$-stably free.
\color{black}

Consider the short exact sequence in cohomology, analogous to the one in homology from~\eqref{eq:SESBranched}:
\[\begin{tikzcd}
	0 & {H^1(\widetilde{D}\times S^1)} & {H^2(\Sigma_d(D))} & {H^2(X_D^d)} & 0.
	\arrow[from=1-1, to=1-2]
	\arrow[from=1-2, to=1-3]
	\arrow[from=1-3, to=1-4]
	\arrow[from=1-4, to=1-5]
\end{tikzcd}\]
We claim that $H^2(X_D^d) \cong_s I$.
We first note that $I^* \cong I$: indeed the free resolution in~\eqref{eq:FreeResolution} yields~$I=\im(t-1)$ and~$\im(\cdot \mathcal{N} \colon \Lambda \to\Lambda) \cong \Lambda_0$, and
dualising $0 \to \im(\cdot \mathcal{N}) \to \Lambda \xrightarrow{t-1} I \to 0$ yields 
$$
\xymatrix{
 \im(\cdot \mathcal{N})^* & \Lambda^* \ar[l]& I^* \ar[l]& 0.  \ar[l] \\
\Lambda_0 \ar[u]^{\cong}& \Lambda \ar[l]_{\aug}\ar[u]^{\cong}&
}
$$
The universal coefficient spectral sequence (in which we use that~$H^i(\Z_d;\Lambda)=0$ for $i>0$) together with Lemma~\ref{lem:StablyAugmentation} then gives 
$ H^2(X_D^d)  \cong \Hom_\Lambda(H_2(X_D^d),\Lambda) \cong_s \Hom_\Lambda(I,\Lambda) \cong I,$ as claimed.

The claim therefore leads to the exact sequence:
\[\begin{tikzcd}
	0 &  \underbrace{H^1(\widetilde{D}\times S^1)}_{\cong \Lambda_0} & {H_2(\Sigma_d(D))} & {\underbrace{H^2(X_D^d)}_{\cong_s I}} & 0.
	\arrow[from=1-1, to=1-2]
	\arrow[from=1-2, to=1-3]
	\arrow[from=1-3, to=1-4]
	\arrow[from=1-4, to=1-5]
\end{tikzcd}\]
We now tensor this sequence over~$\Lambda_1$, yielding
\[\begin{tikzcd}
	{  \underbrace{H^1(\widetilde{D}\times S^1)}_{\cong \Lambda_0 \otimes_\Lambda \Lambda_1}  } & { H_2(\Sigma_d(D)) \otimes_{\Lambda} \Lambda_1} & {\underbrace{H^2(X_D^d) \otimes_{\Lambda} \Lambda_1}_{\cong_s I \otimes_{\Lambda} \Lambda_1}} & 0.
	\arrow["\iota", from=1-1, to=1-2]
	\arrow[from=1-2, to=1-3]
	\arrow[from=1-3, to=1-4]
\end{tikzcd}\]
We assert that the map labelled~$\iota$ is the zero map.
Note that~$\Lambda_0 \otimes_{\Lambda} \Lambda_1 \cong \Z_d$ is~$\Z$-torsion.
Since~$H_2(\Sigma_d(D))$ is projective, it follows that~$H_2(\Sigma_d(D)) \otimes_{\Lambda} \Lambda_1$ is also projective and therefore has no~$\Z$-torsion.
Combining these two facts implies that~$\iota$ is the zero map.

Since~$I$ is free as a~$\Lambda_1$-module, 
it follows that~$H^2(X_D^d)$ is stably free as a~$\Lambda_1$-module.
Since~$\iota$ is the zero map, it follows that~$H_2(\Sigma_d(D)) \otimes_{\Lambda} \Lambda_1$ is also stably free.  
\end{proof}

In order to determine $H_2(\Sigma_d(D))$,  it remains to perform an Euler characteristic calculation.
\begin{lemma}
\label{lem:RankCalculation}
If $M$ is a simply-connected $4$-manifold with boundary $S^3$ and $D \subset M$ is a properly embedded $\Z_d$-disc with $d \neq 0$, then $b_2(\Sigma_d(D))=d \cdot b_2(M).$
\end{lemma}
\begin{proof}
Since $\pi_1(M\setminus D) \cong \Z_d$,  the branched cover~$\Sigma_d(D)$ is simply-connected.
We deduce that~$H_3(\Sigma_d(D)) \cong  H^1(\Sigma_d(D),\partial \Sigma_d(D))=0$.
It follows that $b_2(\Sigma_d(D))+1=\chi(\Sigma_d(D))$.
Next, observe that 
$$ \chi(\Sigma_d(D))
=\chi(X_d(D))+1
=d \cdot \chi(X_D) +1
=d(\chi(X)-1)+1
=d\cdot b_2(X)+1.
$$
Combining these calculations yields $b_2(\Sigma_d(D))=d \cdot b_2(M)$.
\end{proof}

Finally,  we prove the main result of this section.

\begin{proposition}
\label{prop:H2Free}
Let $N$ be a simply-connected $4$-manifold with boundary $S^3$,  let $K \subset \partial N$ be a knot with $H_1(\Sigma_d(K))=0$,  and let~$x \in H_2(N,\partial N)$
be a nonzero class of divisibility $d$.
If~$D \subset N \#^k S^2 \times S^2$ is a disc representing~$x \oplus 0$ with boundary~$K$,  then
$$ H_2(\Sigma_d(D)) \cong \Lambda^{b_2(N)+2k}.$$
\end{proposition}
\begin{proof}
Proposition~\ref{prop:SphereStablyFree} ensures that $H_2(\Sigma_d(D))$ is stably free.
Theorem~\ref{thm:ModulesOverZZd} therefore ensures that it is free, say of rank $r$.
It follows that as an abelian group $H_2(\Sigma_d(D))$ is free of rank $dr$.
Lemma~\ref{lem:RankCalculation} then implies that $r=b_2(N)+2k.$
\end{proof}

\section{Splitting of the form over $\Z$.}
\label{sec:SplittingFormOverZ}

Continuing with Notation~\ref{not:TheUsual}, the goal of this section is to establish the second condition from Proposition~\ref{prop:TopSplitting}, which states that the $\Lambda_0$ version of the pointed hermitian form on $H_2(\Sigma_d(D))$ splits off $k$ hyperbolic summands.
Our overall strategy follows Lee and Wilczy\'nski's approach in the closed case~\cite{LWCommentarii,LWKtheory}.

\medbreak 
We recall some notation. 
We write~$\pi_0\colon \Lambda \to \Lambda_0 =\Lambda/T_{\mathcal{N}}\Lambda\cong\Z$ for the projection map.
We set~$H_2(\Sigma_d(D))_0:=H_2(\Sigma_d(D))/T_{\mathcal{N}}H_2(\Sigma_d(D))$ with the corresponding projection map also denoted~$\pi_0 \colon H_2(\Sigma_d(D)) \to H_2(\Sigma_d(D))_0$. 
Recall that $\lambda_{\Sigma_d(D)}$ descends to a~$\Z$-valued form on~$H_2(\Sigma_d(D))_0$, which we were previously denoting $(\lambda_{\Sigma_d(D)})_0$, but in this section will denote
\begin{align*}
\lambda_0 \colon H_2(\Sigma_d(D))_0 \times H_2(\Sigma_d(D))_0 &\to \Lambda/T_{\mathcal{N}}\Lambda \cong \Z \\
(\pi_0(x),\pi_0(y)) &\mapsto \pi_0(\lambda(x,y)),
\end{align*}
for brevity. Recall that $\lambda_0$ is nonsingular; see Remark~\ref{rem:NonSingular}. 
Here and in the sequel we denote $z_0:=\pi_0(z)$ and $i_* \colon H_2(N) \to H_2(N,\partial N)$ for inclusion induced map.

The next proposition, whose proof closely follows~\cite[Proposition \(2.6\)]{LWGenus} is the main result of this section: it proves the required splitting of~$(H_2(\Sigma_d(D))_0,\lambda_0,z_0)$. 

\begin{proposition}
\label{prop:SplitDownstairs}
The branched covering projection~$p \colon \Sigma_d(D) \to N_k$ induces an isometry
$$p_0 \colon  (H_2(\Sigma_d(D))_0,\lambda_0,z_0) \to (H_2(N_k),Q_{N_k},i_*^{-1}(x) \oplus 0)$$
that satisfies~\(p_0 \circ \pi_0 = p_*\).

In particular, there exists an orthogonal splitting of pointed hermitian modules 
    \[ p_0 \colon (H_2(\Sigma_d(D))_0,\lambda_0,z_0) \xrightarrow{\cong} (H_2(N), Q_N, x) \oplus H(\Z)^k\] which only depends on the branched cover \(p\) and the connected sum decomposition of \(N_k\). 
\end{proposition}
\begin{proof}
    The covering-induced map~$p_* \colon H_2(\Sigma_d(D)) \to H_2(N_k)$ vanishes on~$T_{\mathcal{N}} H_2(\Sigma_d(D))$. 
    Indeed if~$z \in T_{\mathcal{N}}$, then~$d \cdot p_*(z)=p_*(\mathcal{N}z)=0$, but since~$H_2(N_k)$ is torsion-free, we deduce that~$p_*(z)=0$, as asserted.
It follows that~$p_*$ induces a map 
$$ p_0 \colon H_2(\Sigma_d(D))_0 \to H_2(N_k).~$$
By definition, we have~$p_0 \circ \pi_0=p_*$.
Next, we argue that~$p_0$ induces a morphism 
$$ p_0 \colon (H_2(\Sigma_d(D))_0,\lambda_0,z_0) \to (H_2(N_k),Q_{N_k},i_*^{-1}(x) \oplus 0).$$
Clearly,  by definition of~$p_0$ we have~$p_0(z_0)=p_0(\pi_0(z))=p_*(z)=i_*^{-1}(x) \oplus 0.$
The crux is therefore to show that~$p_0$ preserves the intersection forms.
Using the surjectivity of $\pi_0$ and $p_0 \circ \pi_0=p_*$, this amounts to verifying that~$Q_{N_k}(p_*(y),p_*(y'))=\lambda_0(\pi_0(y),\pi_0(y'))$ for all~$y,y' \in H_2(\Sigma_d(D)).$

We use the augmentation-induced isomorphism~$\Lambda/T_{\mathcal{N}}\Lambda\cong\Z$ to identify~$\lambda_0(\pi_0(y),\pi_0(y')) \in \Lambda_0$ with~$\aug(\lambda_0(\pi_0(y),\pi_0(y'))) \in \Z$.
Proposition~\ref{prop:LWIntersectionEquality} now yields
$$\aug(\lambda_0(\pi_0(y),\pi_0(y')))
=\aug\left( \sum_{g \in G} Q_{\Sigma_d(D)}(y,gy')g^{-1} \right)
=\sum_{g \in G} Q_{\Sigma_d(D)}(y,gy')
=Q_{N_k}(p_*(y),p_*(y')).
$$
Thus~$p_0$ preserves the hermitian forms, as claimed.
    
Since $p_* \colon H_2(\Sigma_d(D)) \to H_2(N_k)$ is surjective (see Lemma~\ref{lem:Surjectivep*}, below), and~$p_0 \circ \pi_0=p_*$, it follows that $p_0$ is surjective.
Note that $Q_{N_k}$ and $\lambda_0$ are nonsingular; for the latter we use Remark~\ref{rem:NonSingular}.
Since~$p_0$ is a surjective morphism of nonsingular symmetric forms, it is necessarily injective.
Thus~$p_0$ is an isometry, as required.
\end{proof}

The following fact was needed in the proof above.
\begin{lemma}
\label{lem:Surjectivep*}
The following branched covering induced map is surjective:
$$ p_* \colon H_2(\Sigma_d(D)) \to H_2(N_k).$$
\end{lemma}

\begin{proof}
By excising \(\text{Int}(X_D^d)\cup X_K^d\), we obtain that $H_i(\Sigma_d(D),X_D^d) \cong H_i(\widetilde{D} \times D^2, \widetilde{D} \times S^1)$, and the latter is trivial except in degree~$2$, where it is infinite cyclic.
The same holds for the pair $(N_k, X_D)$ and it follows that the branched covering map induces an isomorphism 
$$H_2(\Sigma_d(D),X_D^d) \xrightarrow{\cong} H_2(N_k,X_D).$$
Since $X_D^d, N_k$ and $\Sigma_d(D)$ are simply-connected, we obtain the following commutative diagram of projection induced maps:
\[\begin{tikzcd}
	0 & {H_2(X_D^d)} & {H_2(\Sigma_d(D))} & {H_2(\Sigma_d(D),X_D^d)} & 0 \\
	0 & {H_2(X_D)} & {H_2(N_k)} & {H_2(N_k,X_D)} & {H_1(X_D)} & 0.
	\arrow[from=1-1, to=1-2]
	\arrow[from=1-2, to=1-3]
	\arrow["{\proj_*}", from=1-2, to=2-2]
	\arrow[from=1-3, to=1-4]
	\arrow["{p_*}", from=1-3, to=2-3]
	\arrow[from=1-4, to=1-5]
	\arrow["\cong", from=1-4, to=2-4]
	\arrow[from=2-1, to=2-2]
	\arrow[from=2-2, to=2-3]
	\arrow[from=2-3, to=2-4]
	\arrow[from=2-4, to=2-5]
	\arrow[from=2-5, to=2-6]
\end{tikzcd}\]
A diagram chase shows that if $\proj_* \colon H_2(X_D^d) \to H_2(X_D)$ is surjective, then so is $p_*$.
To establish the surjectivity of $\proj_*$,  we consider the following diagram which is commutative thanks to the naturality of the Hurewicz map:
\[\begin{tikzcd}
	{\pi_2(X_D^d)} && {\pi_2(X_D)} \\
	{H_2(X_D^d)} && {H_2(X_D).}
	\arrow["{\proj_*}", from=1-1, to=1-3]
	\arrow[from=1-1, to=2-1]
	\arrow[from=1-3, to=2-3]
	\arrow["{\proj_*}", from=2-1, to=2-3]
\end{tikzcd}\]
Since $H_2(\pi_1(X_D)) = H_2(\Z_d)=0$, the right map is surjective and since $X_D^d \to X_D$ is a covering space, the top map is an isomorphism.
It follows that the bottom $\proj_*$ is surjective, as required.
\end{proof}

\section{The evenness condition}
\label{sec:EvenessCondition}

Continuing with Notation~\ref{not:TheUsual}, the goal of this section is to establish the third and final condition from Proposition~\ref{prop:TopSplitting},  namely that~$\lambda_{\Sigma_d(D)}$ is weakly even on the kernel of
\[\begin{tikzcd}
	{\eta \colon H_2(\Sigma_d(D))} & {H_2(\Sigma_d(D))_0} & {H_2(N) \oplus \Z^{2k} } & H_2(N).
	\arrow["{\pi_0}", from=1-1, to=1-2]
	\arrow["{p_0}", from=1-2, to=1-3]
	\arrow["\proj_1",from=1-3, to=1-4]
\end{tikzcd}\]
Here, recall that a hermitian form $(H,\lambda)$ over $\Lambda$ is \emph{weakly even on a subset~$A \subset H$} if for every~$a \in A$,  it satisfies $\lambda(a,a)=r+\overline{r}$ for some~$r \in \Lambda$.
Furthermore, recall that 
$$\pi_0 \colon H_2(\Sigma_d(D)) \to H_2(\Sigma_d(D))_0=H_2(\Sigma_d(D))/T_{\mathcal{N}}H_2(\Sigma_d(D))$$ denotes the canonical projection.
Finally, note the map $p_0$, which was described in Proposition~\ref{prop:SplitDownstairs}, plays the role of the map $\beta_0$ in the splitting theorem (Theorem~\ref{thm:splitting}).

\medbreak

Proving this evenness condition requires several intermediate lemmas.

\begin{notation}
When $d$ is even, we write $T \in \Z_d$ for the unique nontrivial element of order $2$.
\end{notation}

In what follows, we say that an element $p \in \Lambda$ is \emph{even} if it can be written as~$p=r+\overline{r}$ for some~$r \in \Lambda$.
We begin with a criterion ensuring that $\lambda_{\Sigma_d(D)}(y,y)$ is even.

\begin{lemma}
\label{lem:EvenCriterion}
Given~\(y \in H_2(\Sigma_d(D))\), the group ring element~$\lambda_{\Sigma_d(D)}(y,y)$ is even if and only if~$Q_{\Sigma_d(D)}(y,y)$ is even and, in the case $d$ is even,~$Q_{\Sigma_d(D)}(y,Ty)$ is also even.
\end{lemma}
\begin{proof}
By definition of $\lambda_{\Sigma_d(D)}$ we have 
$$\lambda_{\Sigma_d(D)}(y,y)=\sum_{g \in \Z_d}Q_{\Sigma_d(D)}(y,g^{-1}y)g.$$
Since $Q_{\Sigma_d(D)}(y,h^{-1}y)=Q_{\Sigma_d(D)}(hy,y)=Q_{\Sigma_d(D)}(y,hy)$ for every $h\in \Z_d$, we deduce that~$\lambda_{\Sigma_d(D)}(y,y)$ is even if and only if~$Q_{\Sigma_d(D)}(y,gy)$ is even for every $g \in \Z_d$ of order~$2.$
\end{proof}

For the second lemma,  recall that~$\lambda_0$ denotes the form induced by~$\lambda_{\Sigma_d(D)}$ on~$H_2(\Sigma_d(D))_0$.

\begin{lemma}
\label{lem:SumIsEven}
For $x \in H_2(\Sigma_d(D))$,  the following congruences hold:
$$
\lambda_0(\pi_0(x),\pi_0(x))\equiv
\begin{cases}
Q_{\Sigma_d(D)}(x,x) & \quad \text{ if } d \text { is odd} \\
Q_{\Sigma_d(D)}(x,x)+Q_{\Sigma_d(D)}(x,Tx) & \quad \text{ if } d \text { is even}
\end{cases}
\quad \mod 2.
$$
\end{lemma}
\begin{proof}
Using that~$p_0$ is an isometry (Proposition~\ref{prop:SplitDownstairs}),  the equality~$p_0 \circ \pi_0=p_*$ (also from Proposition~\ref{prop:SplitDownstairs}) and Proposition~\ref{prop:LWIntersectionEquality}, we obtain
\begin{equation}
\label{eq:UsefulIntermediate}
\lambda_0(\pi_0(y),\pi_0(y'))
=Q_{N_k}(p_0(\pi_0(y)),p_0(\pi_0(y')))
=Q_{N_k}(p_*(y),p_*(y'))
= \sum_{g \in G} Q_{\Sigma_d(D)}(y,gy').
\end{equation}
The lemma now follows by taking $y=y'$,  reducing modulo $2$,  and recalling that~$Q_{\Sigma_d(D)}(y,h^{-1}y)=Q_{\Sigma_d(D)}(y,hy)$ for every~$h\in \Z_d$.
\end{proof}

Our final lemma is specific to the case where $d$ is even.

\begin{lemma}
\label{lem:Edmonds}
Assume that $d$ is even.
If there exists an~$x \in H_2(\Sigma_d(D);\Z_2)$ with $Q_{\Sigma_d(D)}(x,Tx) \neq 0$, then, for every $y \in H_2(\Sigma_d(D))$
$$ Q_{\Sigma_d(D)}(y,Ty) = Q_{\Sigma_d(D)}(y,z) \mod 2.$$
\end{lemma}
\begin{proof}
By hypothesis,  there is a class~$\alpha \in H^2(\Sigma_d(D),\partial \Sigma_d(D);\Z_2)$ with~$\alpha \cup T^*\alpha \neq 0$.
Writing~$k \colon (D^2,\partial D^2) \to (\Sigma_d(D),\partial \Sigma_d(D))$ for the embedding with image~$\widetilde{D}$, a relative version of~\cite[Proposition 5.1]{EdmondsAspects} (see Proposition~\ref{prop:BredonRel}) ensures that for every $\beta \in H^2(\Sigma_d(D),\partial \Sigma_d(D);\Z_2)$, we have
\begin{equation}
\label{eq:RelativeEdmonds}
\langle \beta \cup T^*\beta,[\Sigma_d(D)] \rangle =\langle k^*(\beta),[D^2]\rangle \mod 2.
\end{equation}
Since~$i_*(z)=[\widetilde{D}]=k_*([D^2]) \in H_2(\Sigma_d(D),\partial \Sigma_d(D))$,  using the definition of the intersection form, it follows that for all $y \in H_2(\Sigma_d(D))$ 
\begin{align*}
Q_{\Sigma_d(D)}(z,y)
&=\langle \PD_{\Sigma_d(D)}^{-1}(i_*(y)),z\rangle
=\langle i^*(\PD_{\Sigma_d(D)}^{-1}(y)),z\rangle
=\langle \PD_{\Sigma_d(D)}^{-1}(y),k_*([D^2])\rangle\\
&=\langle k^*(\PD_{\Sigma_d(D)}^{-1}(y)),[D^2]\rangle 
=\langle \PD_{\Sigma_d(D)}^{-1}(y)  \cup T^*\PD_{\Sigma_d(D)}^{-1}(y) ,[\Sigma_d(D)]\rangle  \\
&=Q_{\Sigma_d(D)}(y,Ty) \mod 2.
\end{align*}
Here we used~\eqref{eq:RelativeEdmonds} in the fifth equality.
This concludes the proof of the lemma.
\end{proof}

We now prove the main result of this section, namely the evenness condition of Proposition~\ref{prop:TopSplitting}.

\begin{proposition}
\label{prop:EvenessCondition}
If~$D \subset N_{k}$ is a~$\Z_d$-surface representing~$x \oplus 0 \in H_2(N_k, \partial N)$, then~$\lambda_{\Sigma_d(D)}$ is weakly even on
$$\operatorname{Ker}:=\ker \Big( H_2(\Sigma_d(D)) \xrightarrow{\pi_0} H_2(\Sigma_d(D))_0 \xrightarrow{p_0,\cong}  H_2(N) \oplus H_2(S^2 \times S^2)^k \xrightarrow{\operatorname{proj}} H_2(N) \Big).$$
\end{proposition}
\begin{proof}
By Lemma~\ref{lem:EvenCriterion}, it suffices to prove that $Q_{\Sigma_d(D)}(y,y)$ (as well as $Q_{\Sigma_d(D)}(Ty,y)$ if $d$ is even) is even for every $y \in \operatorname{Ker}$.
Now given $y \in \operatorname{Ker}$, note that~$p_0(\pi_0(y))\in H_2(S^2\times S^2)$. 
But elements of~$H_2(\#^k S^2\times S^2)\subset  H_2(N_k)$ pair trivially with themselves and also trivially with all elements~$w\in H_2(N)$, under $Q_{N_k}$. 
In particular since $p_0$ is an isometry (Proposition~\ref{prop:SplitDownstairs}),
we have~$\lambda_0(\pi_0(y),\pi_0(y))=Q_{N_k}(p_0(\pi_0(y)),p_0(\pi_0(y')))=0$.
For $d$ odd,  we apply Lemma~\ref{lem:SumIsEven} to compute that $Q_{\Sigma_d(D)}(y,y) \equiv \lambda_0(\pi_0(y),\pi_0(y))=0$ is even for every $y \in \operatorname{Ker}$, establishing the proposition in this case.

Now assume that~$d$ is even and write $x=dx_0$ with $x_0 \in H_2(N,\partial N)$ primitive.
Since~$H_2(\Sigma_d(D))$ is $\Lambda$-free (recall Proposition~\ref{prop:H2Free}),  there is a $z_0 \in H_2(\Sigma_d(D))$ with $z=\mathcal{N}z_0.$
By definition of~$z$ we have~$p(z)=i_*^{-1}(x)$,  and thus~$dx_0=i_*^{-1}(x)=p(z)=d  (p_0 \circ \pi_0 (z_0))$.
Dividing by~$d$ we deduce that~$x_0=p_0(\pi_0(z_0))$.
Since $p_0$ is an isometry, the previous paragraph therefore ensures that $\lambda_0(\pi_0(y),\pi_0(z_0))
=Q_{N_k}(p_0(\pi_0(y)),p_0(\pi_0(z_0))=0.$
It follows that~$Q_{\Sigma_d(D)}(y,Ty)$ is even for every~$y \in \operatorname{Ker}$; to see this, consider that if there existed~$y \in \operatorname{Ker}$ violating this parity requirement, combining Lemma~\ref{lem:Edmonds} with~\eqref{eq:UsefulIntermediate} would lead to the following contradiction:
\begin{align*}
1 
&= Q_{\Sigma_d(D)}(y,Ty)
=Q_{\Sigma_d(D)}(y,z)
=\sum_{g \in \Z_d} Q_{\Sigma_d(D)}(y,gz_0) 
=\lambda_0(\pi_0(y),\pi_0(z_0)) 
 = 0 \quad  \mod 2.
 \end{align*}
 Thus $\lambda_0(\pi_0(y),\pi_0(y))$ and~$Q_{\Sigma_d(D)}(Ty,y)$ are both even.
Applying Lemma~\ref{lem:SumIsEven}, we compute that~$Q_{\Sigma_d(D)}(y,y)$ is also even,  thus concluding the proof in the even case as well.
\end{proof}

\begin{remark}
In the case of spheres in closed $4$-manifolds,  Lee-Wilczy\'nski assert that Lemma~\ref{lem:Edmonds} holds without having to assume that there exists an~$x \in H_2(\Sigma_d(D);\Z_2)$ with $Q_{\Sigma_d(D)}(x,Tx) \neq 0$; see~\cite[Equation (4.4)]{LWKtheory}.
We have not been able to confirm this assertion. 
Working in the smooth category, one might attempt to improve the sketch above~\cite[Proposition 5.1]{EdmondsAspects} to a proof of their assertion. 
For this,  one might work in the base space of the branched cover in order to overcome issues of equivariant transversality present in Edmonds' sketch. 
As we have used an alternative argument to prove Proposition~\ref{prop:EvenessCondition}, we have not attempted to complete these details (which would moreover need to be performed in the topological category, using a locally linear involution).
Finally, we note that when $Q_{\Sigma_d(D)}(z,z) \neq 0$, Lemma~\ref{lem:Edmonds} does automatically hold: in this case,~$Q_{\Sigma_d(D)}(z,Tz) \neq 0$ and so the lemma follows from~\cite[Proposition 5.1]{EdmondsAspects}.
\end{remark}

\section{Proof of the main theorem}
\label{sec:ProofMain}

We recall the main theorem for the reader's convenience.
\begin{customthm}{\ref{thm:Main}}
\label{thm:MainProof}
Let \(N\) be a compact, oriented, simply-connected $4$-manifold with boundary \(S^3\). Let $x \in H_2(N, \partial N)$ be a nonzero class of divisibility $d$, and 
suppose \(K\subset S^3\) is a knot such that~\(H_1(\Sigma_d(K))=0\). 
The following are equivalent:
\begin{itemize}
\item The knot $K$ is sliced by a simple disc in $N$ representing $x$.
\item
\begin{enumerate}
\item \label{item:1Main} $\operatorname{Arf}(K)+\ks(N)+\frac{1}{8}(\sigma(N)-x \cdot x) \equiv 0 \ \text{ mod } 2$ \ if $x$ is characteristic, and
\item\label{item:2Main}  $b_2(N)\geq \max_{0\leq j<d}|\sigma (N)-\frac{2j(d-j)}{d^2}\, x\cdot x +\sigma_K (e^{\frac{2\pi i j}{d}})|$.
\end{enumerate}
\end{itemize}
\end{customthm}
\begin{proof}
The necessity of the conditions is a consequence of Theorem~\ref{thm:StableWithBoundaryIntro} and Proposition~\ref{prop:SignatureCondition}.
We turn to the converse.
Theorem~\ref{thm:StableWithBoundaryIntro} produces an embedded disc~$D \subset N\#^k S^2 \times S^2$ with boundary~$K$ that represents the class $x \oplus 0$. 
Proposition~\ref{prop:TopSplitting} (which applies thanks to Propositions~\ref{prop:H2Free}, \ref{prop:SplitDownstairs} and~\ref{prop:EvenessCondition}) then ensures that $D$ can be upgraded to a disc in~$N$ with boundary~$K$ and representing~$x$.
\end{proof}

\appendix
\section{Intersection forms and transfers}

The goal of this section is to prove an equality that is implicit in~\cite[Proposition~3.7]{LWCommentarii}.
As we will see,  this result (namely Proposition~\ref{prop:LWIntersectionEquality}) follows fairly readily from properties of transfer homomorphisms.
However, since the literature on transfers is difficult to navigate, we instead decided to do things by hand.

\begin{notation}
We fix the following notation.
Set $G:=\Z_d$ with $d \neq 0$,  let $X$ be a simply-connected $4$-manifold whose boundary is empty or $S^3$,  let $F \subset X$ be a properly embedded surface with knot group~$\pi_1(X_F) \cong G$, and let $p \colon \Sigma(F) \to X$ be the corresponding branched cover.
Here we assume that $F$ has a boundary if and only if $X$ has a boundary.
\end{notation}

\begin{lemma}
\label{lem:TransferFundamentalClass}
The following equality holds:
$$p_*([\Sigma(F)])=|G|\cdot[X] \in H_4(X,\partial X).$$
\end{lemma} 
\begin{proof}
We claim the lemma reduces to the closed case.
Double the pair~$(X,F)$ to obtain~$(V,S)$.
As~$X$ is simply-connected and~$\pi_1(X_F)\cong G$, we obtain that~$V$ is simply-connected and~$\pi_1(V_S)\cong G$.
In particular the branched cover~$\Sigma(S)$ is simply-connected, whence~$H_3(\Sigma(S))=0$.
Consider the following commutative diagram:
$$
\xymatrix@R0.4cm{
0 \ar[r]&H_4(\Sigma(S)) \ar[r]^-{\cong}\ar[d]^{p_*}_{\cdot |G|}& {\overbrace{H_4(\Sigma(S),\Sigma(F))}^{\cong H_4(\Sigma(F),\partial\Sigma(F))}}\ar[d]^{p_*}_{} \ar[r]&0 \\
0 \ar[r]&H_4(V) \ar[r]^-\cong& {\underbrace{H_4(V,X)}_{\cong H_4(X,\partial X)}} \ar[r]&0. \\
}
$$
The left map is multiplication by $|G|$ because we are assuming the lemma holds in the closed case.
This concludes the proof of the claim.

We prove the lemma when $F$ and $X$ are closed.
Consider the following commutative diagram:
$$
\xymatrix@R0.4cm{
0 \ar[r]&H_4(\Sigma(F)) \ar[r]^-{\cong}\ar[d]^{p_*}& {\overbrace{H_4(\Sigma(F),\overline{\nu}(F))}^{\cong H_4(\widetilde{X}_F,\partial \widetilde{X}_F)}}\ar[d]^{p_*}_{\cdot |G|} \ar[r]&0 \\
0 \ar[r]&H_4(X) \ar[r]^-\cong& {\underbrace{H_4(X,\overline{\nu}(F))}_{\cong H_4(X_F,\partial X_F)}} \ar[r]&0. \\
}
$$
The right hand side map is multiplication by $|G|$ because (unbranched) regular $G$-covering projections have degree $|G|$.
\end{proof}

For the next statement, recall that $\mathcal{N} \in \Z[G]$ denotes the norm element.

\begin{proposition}
\label{prop:McGyver}
For every $y \in H_2(\Sigma(F))$, the homomorphism~$\tr:=\PD_{\Sigma(F)} \circ p^* \circ \PD_{X}^{-1}$ satisfies
$$\tr \circ p_*(y)=\mathcal{N}y \in H_2(\Sigma(F)).$$
\end{proposition}
\begin{proof}
We use
Lemma~\ref{lem:TransferFundamentalClass} to verify that~$p_* \circ \tr(\sigma)=|G|\sigma$:
\begin{align*}
p_* \circ \tr(\sigma)
&=p_*(p^*(\PD_X^{-1}(\sigma)) \cap [\Sigma(F)]) 
=\PD_X^{-1}(\sigma) \cap p_*([\Sigma(F)]) 
=\PD_X^{-1}(\sigma) \cap |G|[X] 
=|G|\sigma.
\end{align*}
Next, we observe that given $\sigma \in H_2(X),$ every $g \in G$ fixes $\tr(\sigma)$.
Indeed since deck transformations are orientation-preserving homeomorphisms,
 they are degree one and so
\begin{align*}
g_* \circ \tr(\sigma)
&=g_*(g^*g^{-*}p^*(\PD_X^{-1}(\sigma)) \cap [\Sigma(F)]) 
=g^{-*}p^*(\PD_X^{-1}(\sigma)) \cap g_*([\Sigma(F)]) \\
&=p^*(\PD_X^{-1}(\sigma)) \cap [\Sigma(F)] 
=\tr(\sigma).
\end{align*}
Since $H_2(\Sigma(F))$ is free abelian, it suffices to prove that $\mathcal{N}y=\tr \circ p_*(y) \in H_2(\Sigma(F);\Q)$.
Both elements lie in $H_2(\Sigma(F);\Q)^G$.
With rational coefficients, the projection induces an isomorphism~$p_* \colon H_2(\Sigma(F);\Q)^G \to H_2(X;\Q)$~\cite[Chapter III Theorem 7.4]{BredonTopology}.
Since we calculated that~$p_*(\mathcal{N}y)=|G|p_*(y)=p_* \circ \tr \circ p_*(y)$, we obtain~$\mathcal{N}y=\tr \circ p_*(y)$, as claimed. 
\end{proof}

\begin{proposition}
\label{prop:LWIntersectionEquality}
The following equality holds for every~$x,y\in H_2(\Sigma(F))$:
$$\sum_{g \in G} Q_{\Sigma(F)}(x,gy)= Q_{X}(p_*(x),p_*(y)).$$
\end{proposition}
\begin{proof}
Consecutively use Proposition~\ref{prop:McGyver}, the definition of  the intersection form and the properties of Poincar\'e duality to obtain
\begin{align*}
\sum_{g \in G} Q_{\Sigma(F)}(x,gy)
&=Q_{\Sigma(F)}(x,\mathcal{N}y)
=Q_{\Sigma(F)}(x,\tr \circ p_*(y)) 
=Q_{\Sigma(F)}(x, \PD_{\Sigma(F)} \circ p^* \circ \PD_{X}^{-1} \circ p_*(y)) \\
&=\langle  \PD_{\Sigma(F)}^{-1} \circ j_*(\PD_{\Sigma(F)}  \circ p^* \circ \PD_{X}^{-1} \circ p_*(y)),x\rangle 
=\langle  j^* \circ p^* \circ \PD_{X}^{-1} \circ p_*(y),x\rangle \\
&=\langle  p^* \circ \PD_{X}^{-1} \circ j_* \circ p_*(y),x\rangle 
=\langle  \PD_{X}^{-1} \circ j_*\c(p_*(y)),p_*(x)\rangle 
=Q_X(p_*(x),p_*(y)).
\end{align*}
The sixth equality is best summarised by the commutativity of the following diagram:
$$
\xymatrix@R0.5cm{
H_2(X) \ar[r]^-{\PD_X^{-1}}\ar[d]^-{j_*}&H^2(X,\partial X) \ar[r]^-{p^*}\ar[d]^{j^*}& H^2(\Sigma,\partial \Sigma)\ar[d]^-{j^*} \\
H_2(X,\partial X) \ar[r]^-{\PD_X^{-1}}&H^2(X) \ar[r]^-{p^*}& H^2(\Sigma).
}
$$
This concludes the proof of the proposition.
\end{proof}

\section{A relative version of Bredon-Edmonds' result}

In this section, we state and prove a version of Edmonds' result~\cite[Proposition 5.1]{EdmondsAspects}. Edmonds' argument is made in the absolute case and we need a version in the case of manifolds with boundary and properly embedded surfaces.

The Edmonds result crucially uses a result of Bredon~\cite[Theorem~VII.7.4]{BredonIntroduction}, which is also only stated in an absolute version. Bredon's result has the following relative variant. The lengthy, but ultimately formal, exercise of adjusting Bredon's proof to the relative case will not be included here, and is left to the reader.

\begin{theorem}[Bredon]
\label{thm:Bredonmainrelative}
Let~$T$ be an involution on a paracompact space~$X$ and write $F$ for the fixed point set. Let~$A\subset X$ be a subspace such that $T(A)\subset A$.
Suppose~$H^i(X,A;\Z_2)=0$ for~$i>2n$ and~$T^*$ is the identity on~$H^{2n}(X,A;\Z_2)$. 
If~$a\in H^n(X,A;\Z_2)$ is such that~$a\cup T^* a\in H^{2n}(X,A;\Z_2)$ is non-zero, then~$k^*(a)\in H^n(F,F\cap A;\Z_2)$ is non-zero,  where $k \colon F \to X$ denotes the inclusion.
\end{theorem}

We now recall Edmonds' argument, with the minor adjustments for the relative case.

\begin{proposition}[Edmonds]
\label{prop:BredonRel}
Let~$T$ be an orientation-preserving involution on a connected orientable~$2n$-manifold~$X$ with non-empty boundary $A$. Suppose the fixed-point set is a proper submanifold with boundary~$(F,\partial F)\subset (X,A)$. Suppose there exists~$a\in H^n(X,A;\Z_2)$ such that~$a\cup T^* a\in H^{2n}(X,A;\Z_2)$ is non-zero. Then
\[
\langle c\cup T^*c,[X]\rangle\equiv \langle k^*(c),[F]\rangle\mod 2
\]
for all~$c\in H^n(X,A;\Z_2)$, where~$k\colon F\to X$ denotes the inclusion of the fixed-point set of~$T$.
\end{proposition}
\begin{proof}
Define two functions
\[
\begin{array}{rcll}
\varphi\colon H^n(X,A;\Z_2)&\to&\Z_2; \qquad\qquad&\varphi(c)=\langle c\cup T^* c,[X]\rangle\\
\psi\colon H^n(X,A;\Z_2)&\to&\Z_2; &\psi(c)=\langle k^*(c),[F]\rangle.
\end{array}
\]
As~$T$ is orientation-preserving, it acts trivially on~$H^{2n}(X,A;\Z_2)$ and we may apply Theorem~\ref{thm:Bredonmainrelative}: if~$\varphi(c)\neq0$, then~$\psi(c)\neq 0$. 
Now, by hypothesis,~$a\cup T^* a\neq 0$, and so by Poincar\'{e} duality~$\varphi$ is not the trivial map. 
Hence~$\psi$ is not the trivial map. 
The kernel of a surjective map to~$\Z_2$ has index 2, hence both maps vanish on exactly half the elements of~$H^n(X,A;\Z_2)$. 
But we already know that the maps agree on at least half the elements (the ones where~$\varphi(c)=1$). 
Hence they must agree on all the elements.
\end{proof}

\bibliographystyle{myamsalpha}
\bibliography{BiblioLW}

\end{document}